\documentclass[envcountsect,referee]{svjour3}
\smartqed  
\usepackage{graphicx}
\usepackage{latexsym,bm,amsmath,amssymb,mathrsfs}
\usepackage{rotating}
\usepackage{tabularx}
\usepackage{multirow,booktabs,cases}
\usepackage[misc]{ifsym}
\usepackage[style=1]{mdframed}
\usepackage{algorithm,algorithmic}
\usepackage[colorlinks=true]{hyperref}
\hypersetup{urlcolor=blue,citecolor=blue,linkcolor=blue}
\usepackage{epstopdf}
\usepackage{pifont}
\usepackage{enumitem,float}

\usepackage{diagbox,adjustbox}

\usepackage{ulem, soul, color, tikz, marginnote}

\usepackage[justification=centering]{caption}

\def\[{\begin{equation}}
\def\]{\end{equation}}

\newcommand{\ii}{\mathbf{i}}
\newcommand{\jj}{\mathbf{j}}
\newcommand{\kk}{\mathbf{k}}

\numberwithin{equation}{section}
\allowdisplaybreaks

\begin{document}

\graphicspath{{./FIG/}}


\title{A metric function for dual quaternion matrices and related least-squares problems}

\titlerunning{A metric function for dual quaternion matrices and related least-squares problems}     

\author{Chen Ling \and Chenjian Pan \and Liqun Qi
}



\institute{C. Ling \at
Department of Mathematics, Hangzhou Dianzi University, Hangzhou, 310018, China.\\
\email{macling@hdu.edu.cn}
\and   C.J. Pan\at School of Mathematics and Statistics, Ningbo University,
Ningbo, 315211, China. \\
\email{panchenj@163.com}
\and L.Q. Qi \at Department of Applied Mathematics, The Hong Kong Polytechnic University, Hung Hom,
    Kowloon, Hong Kong, China \\ \email{maqilq@polyu.edu.hk}
}

\date{Received: date / Accepted: date}

\maketitle

\begin{abstract}
Solving dual quaternion equations is an important issue in many fields such as scientific computing and engineering applications. In this paper, we first introduce a new metric function for dual quaternion matrices. Then, we reformulate dual quaternion overdetermined equations as a least squares problem, which is further  converted into a bi-level optimization problem. Numerically, we propose two implementable proximal point algorithms for finding approximate solutions of dual quaternion overdetermined equations. The relevant convergence theorems 
have also been established. Preliminary simulation results on synthetic and color image datasets demonstrate the effectiveness of the proposed algorithms.

\keywords{Quaternion \and dual quaternion matrix\and metric function\and least squares problem\and Jacobi-proximal
ADMM algorithm\and convergence.}



\end{abstract}
\section{Introduction}\label{Sect1-Introd}

Quaternions, which were introduced by the Irish mathematician Hamilton \cite{Ha43} in 1843, are extensions of complex numbers.  A quaternion has the form $u = u_0 + u_1\ii + u_2\jj + u_3\kk$, where $u_0, u_1, u_2,u_3$ are real numbers, $\ii, \jj$ and $\kk$ are three imaginary units of quaternions, satisfying $\ii^2 = \jj^2 = \kk^2 =\ii\jj\kk = -1$, $\ii\jj = -\jj\ii = \kk$, $\jj\kk = - \kk\jj = \ii$ and $\kk\ii = -\ii\kk = \jj$. The multiplication of quaternions satisfies the distribution law, but is noncommutative. Nowadays, quaternions have been widely applied in many engineering fields such as quantum mechanics and color image processing, e.g., see \cite{DSR23,EBS14,FJSS62,Gan18,JNS19,QLWZ22,Sa96,ZKW16} and reference therein.

Dual quaternions were introduced by Clifford \cite{Cl73} in 1873, and have become one of the core knowledge of Clifford algebra or geometric algebra. 
A dual quaternion has the form $\dot{q} = q_{\rm st} + q_{\rm in}\epsilon$, where $q_{\rm st}, q_{\rm in}$ are quaternions, which are the standard part and the infinitesimal part of $\dot{q}$, respectively.  Here, $\epsilon$ is the dual unit subjected to the rules $\epsilon\neq 0$, $0\epsilon=\epsilon 0=0$, $1\epsilon =\epsilon1 =\epsilon$ and $\epsilon^2=0$. 
In mechanics, the dual quaternions are applied as a number system to represent rigid transformations in three dimensions. Similar to the way that rotations in 3D space can be represented by quaternions of unit length, rigid motions in 3D space can be represented by dual quaternions of unit length. Because dual quaternions can be used to represent coordinately a combination of rigid body's rotation and displacement, they have many applications in engineering fields, such as 3D computer graphics, robotics control and computer vision \cite{BLH19,CKJC16,Da99,MKO14,MXXLY18,TRA11,WYL12,WZ14}. Dual quaternions employed as a novel data structure are ubiquitous across a broad range of fields from kinematics and statics to dynamics \cite{Fis17}.  
The study of dual quaternion matrices and their applications in formation control in 3D space can be traced back to the Ph.D. thesis of X. Wang \cite{Wx11} in 2011. In \cite{WYZ}, Wang, Yu and Zheng studied several dual quaternion matrices in multi-agent formation control. Recently, Qi et al used dual quaternion matrices to study undirected and directed gain graph and (multi-agent) formation control \cite{Qi-Cui24,QC24,QWL23}. In addition, by utilizing dual quaternion matrix theory, many scholars also studied various  engineering problems such as hand-eye calibration \cite{CLQY24,Qi23} and simultaneous localization and mapping (SLAM) \cite{CCHQ24}. Considering the wide application background of dual quaternion matrices, more researchers pay attention to the fundamental theoretical properties of dual quaternion matrices, including the SVD \cite{QiLu23}, singular values and low rank approximations \cite{LHQ22}, minimax principle and generalized inverses \cite{LQY23}, spectral norm and trace \cite{LHQF24}, determinant \cite{CQSW24}, and algorithms for solving eigenvalues of dual quaternion Hermitian matrices \cite{CQ24}.

The least squares problem is a class of important methods for solving approximate solutions of overdetermined equations (in this case, we cannot obtain an exact solution to these equations), and plays a key role in engineering fields such as compressive sensing, image processing, and data recovery. Although quaternion multiplication does not satisfy commutative laws, its least squares problem still possesses some important properties similar to classical least squares problems. For example, let $A$ and $B$ be the quaternion matrices of appropriate size, consider the quaternion least squares problem in the following form:
\begin{equation}\label{QSLP-Prob1}
\displaystyle\mathop{\rm min}\limits_{X\in \mathbb{Q}^{n\times p}}~\displaystyle\frac{1}{2}\|AX-B\|_F^2,
\end{equation}
where $\mathbb{Q}^{n\times p}$ denotes the set of all $n\times p$ quaternion matrices. Its solution set can be expressed as
\begin{equation}\label{QSLP--solution}
\mathcal{X}^\diamond=\big\{X=A^\dag B+(I_n-A^\dag A)Z~:~Z\in \mathbb{Q}^{n\times p}\big\},
\end{equation}
where $A^\dag$ denotes the Moore-Penrose inverse of the quaternion matrix $A$ (see Theorem 3.2.1 and Definition 1.6.1 in \cite{WLZZ18}). By combining the quaternion least squares method with some related matrix factorization or low-rank/sparse regularization techniques, effective algorithms can be designed for color image and video data processing problems. Many algorithms have been widely applied in the engineering and scientific computing fields, as shown in \cite{WWF08,YBM08,ZWLZ15} and reference therein. Just as the dual numbers are applied to large-scale brain functional magnetic resonance imaging (MRI) data processing \cite{WDW24}, in many real applications, the infinitesimal parts of dual quaternions can be regarded as perturbations or noise to their standard parts as well as derivatives concerning time for time series data. This excellent property provides us with the possibility to use dual quaternion matrices to express relevant data and to design more effective algorithms, among which how to effectively solve the dual quaternion least squares problem is a key step. Due to the different priority given to the standard and infinitesimal parts of dual numbers in their total order definition, solving dual quaternion least squares problem is a challenge. In this paper, we first introduce a new metric function for dual quaternion matrices. Then, we gainfully employ the new metric function to formulate the dual quaternion overdetermined equation as a least squares problem, which will be further converted into an equivalent bi-level optimization problems. Numerically, we accordingly introduce two implementable proximal point algorithms to find numerical solutions of the problem under consideration. Some computational results on synthetic and color image datasets support the ideas of this paper.

The rest of the paper is organized as follows. Following a brief recalling of dual numbers, quaternions, and dual quaternions in Section \ref{Sect2-Prelim}, a dual quaternion matrix
metric function is introduced in Section \ref{Metric Fun}. In Section \ref{Sect4-0}, we consider the approaches for solving two dual quaternion least squares problems. Based upon converting the involved problems into {equivalent} bi-level optimization models,  two proximal (alternating) minimization algorithms are proposed. Moreover, the convergence of the proposed algorithms {is} analyzed in Section \ref{Conanalysis}. In Section \ref{NumTest}, we conduct simulation experiments to evaluate the performance of the proposed approaches. Finally, we give some concluding remarks to complete this paper in Section \ref{FinalRe}.

\section{Preliminaries}\label{Sect2-Prelim}

\subsection{Dual numbers}\label{Sect2.1}
Denote by $\dot{\mathbb R}$ the set of dual numbers. A dual number $\dot{q}\in \dot{\mathbb{R}}$ has the form $\dot{q} = q_{\rm st} + q_{\rm in}\epsilon$, where $q_{\rm st},q_{\rm in}\in \mathbb{R}$. 
We call $q_{\rm st}$ the standard part of $\dot{q}$, and $q_{\rm in}$ is the infinitesimal part of $\dot{q}$.  The infinitesimal unit $\epsilon$ is commutative in multiplication with real numbers. 
 If $q_{\rm st} \not = 0$, we say that $\dot{q}$ is appreciable; otherwise, we say that $\dot{q}$ is infinitesimal. In \cite{QLY21}, a total order  for dual numbers was introduced. For $\dot{p} = p_{\rm st} + p_{\rm in}\epsilon, \dot{q} = q_{\rm st} + q_{\rm in}\epsilon \in \dot{\mathbb R}$, we have $\dot{q} < \dot{p}$ if $q_{\rm st} < p_{\rm st}$, or $q_{\rm st} = p_{\rm st}$ and $q_{\rm in} < p_{\rm in}$.  
 Thus, if $\dot{q} > 0$, we say that $\dot{q}$ is a positive dual number; and if $\dot{q} \ge 0$, we say that $\dot{q}$ is a nonnegative dual number. 
For given $\dot{p} = p_{\rm st} + p_{\rm in}\epsilon, \dot{q} = q_{\rm st} + q_{\rm in}\epsilon \in \dot{\mathbb R}$, we denote
$\dot{p} + \dot{q} =p_{\rm st}+q_{\rm st} +(p_{\rm in}+q_{\rm in})\epsilon$ and $\dot{p}\dot{q} = p_{\rm st}q_{\rm st} +(p_{\rm st}q_{\rm in}+p_{\rm in} q_{\rm st})\epsilon$.
The absolute value \cite{QLY21} of $\dot{q} \in \dot{\mathbb R}$ is defined by
\begin{equation*}
|\dot{q}| = \left\{ \begin{array}{ll}|q_{\rm st}| + \displaystyle\frac{q_{\rm st}}{|q_{\rm st}|}q_{\rm in}\epsilon, & {\rm if~}  q_{\rm st} \not = 0, \\
|q_{\rm in}|\epsilon, & {\rm otherwise}.  \end{array}  \right.
\end{equation*}

\subsection{Quaternion, quaternion matrices and quaternion matrix functions}\label{Sect2.2}
Denote by $\mathbb{Q}$ the set of all quaternions. 
For given $u = u_0+ u_1\ii + u_2\jj + u_3\kk\in \mathbb{Q}$, the conjugate of $u$ is $\bar u := u_0- u_1\ii - u_2\jj - u_3\kk$. 
It is easy to verify that $\overline{uv}=\bar{v}\bar u$ for any $u,v\in \mathbb{Q}$. For given $u = u_0 + u_1\ii + u_2\jj + u_3\kk\in \mathbb{Q}$, the norm of $u$ is defined as $|u|:=\sqrt{\bar uu}=\sqrt{u_0^2+u_1^2+u_2^2+u_3^2}$.

Denote by $\mathbb{Q}^{m\times n}$ the collection of all $m\times n$ matrices with quaternion entries. Specifically, denote by $\mathbb{Q}^m$ the collection of all column vectors with $m$ components, i.e., $\mathbb{Q}^m=\mathbb{Q}^{m\times 1}$. We denote the quaternion column vectors by boldfaced lowercase letters (e.g., ${\bf u,v,\ldots}$), and denote the quaternion matrices by capital letters (e.g., $U,V,\ldots$). It is clear that any quaternion matrix $U\in\mathbb{Q}^{m\times n}$ can be expressed as $U=U_0+U_1\ii+U_2\jj+U_3\kk$, where $U_1, U_2, U_2, U_3\in \mathbb{R}^{m\times n}$. For given $U=(u_{ij})\in \mathbb{Q}^{m\times n}$, the transpose of $U$ is denoted as $U^\top=(u_{ji})$, the conjugate of $U$ is denoted as $\bar U = (\bar u_{ij})$, and the conjugate transpose of $U$ is denoted as $U^*=(\bar u_{ji})=\bar U^\top$. A square matrix $U\in \mathbb{Q}^{m\times m}$ is called Hermitian if $U^*=U$. A quaternion Hermitian matrix $Q\in \mathbb{Q}^{m\times m}$ is called positive definite, if ${\bf x}^*Q{\bf x}>0$ for any nonzero ${\bf x}\in \mathbb{Q}^m$. For given $U=(u_{ij}),~V=(v_{ij})\in \mathbb{Q}^{m\times n}$, denote by $\langle U,V\rangle$ the quaternion-valued inner product, i.e., $\langle U,V\rangle=\sum_{i=1}^m\sum_{j=1}^n\bar v_{ij}u_{ij}$, and denote by $\langle U,V\rangle_R:=\frac{1}{2}(\langle U,V\rangle+\langle V,U\rangle)$ the real-valued inner product of $U$ and $V$. 
It is obvious that $\langle U,V\rangle={\rm trace}(V^*U)$ and $\langle U,V\rangle_R=\frac{1}{2}({\rm trace}(V^*U)+{\rm trace}(U^*V))$, where ${\rm trace}(A)=\sum_{i=1}^n a_{ii}$ for any $A=(a_{ij})\in \mathbb{Q}^{n\times n}$. For any $U=(u_{ij})\in \mathbb{Q}^{m\times n}$, the Frobenius norm of $U$ is defined by
$$
\|U\|_F=\sqrt{\langle U,U\rangle}=\sqrt{\sum_{i=1}^m\sum_{j=1}^n|u_{ij}|^2}.
$$
It is easy to verify that
$\|U\|_F=\sqrt{\|U_0\|^2+\|U_1\|^2+\|U_2\|^2+\|U_3\|^2}$,
where $U=U_0+U_1\ii+U_2\jj+U_3\kk\in \mathbb{Q}^{m\times n}$. For given positive definite matrix $Q\in \mathbb{Q}^{m\times m}$, define $\|U\|_Q=\sqrt{\langle QU,U\rangle}=\sqrt{{\rm trace}(U^* QU)}$ for $U\in \mathbb{Q}^{m\times n}$. It is easy to verify that $\|\cdot\|_Q$ is a norm on $\mathbb{Q}^{m\times n}$, i.e., $\|\cdot\|_Q$ satisfies: i) $\|U\|_Q\geq 0$ for any $U\in \mathbb{Q}^{m\times n}$, and $\|U\|_Q=0$ if and only if $U=0$; ii) $\|U\alpha\|_Q=|\alpha|\|U\|_Q$ for any $U\in \mathbb{Q}^{m\times n}$ and $\alpha\in \mathbb{Q}$; and iii) $\|U+V\|_Q\leq \|U\|_Q+\|V\|_Q$ for any $U,V\in \mathbb{Q}^{m\times n}$.

\begin{proposition}\label{QM-Ineq}
 Let $A\in \mathbb{Q}^{m\times p}$ and $B\in \mathbb{Q}^{p\times n}$. If $m\geq p$, then it holds that $\sigma_{\rm min}(A)\|B\|_F\leq \|AB\|_F$, where $\sigma_{\rm min}(A)$ is the smallest singular value of $A$.
 \end{proposition}

 \begin{proof}
 By Theorem 7.2 in \cite{Zh97}, there exist unitary matrices $U\in \mathbb{Q}^{m\times p}$ and $V\in \mathbb{Q}^{p\times p}$, such that $A=U \Sigma V^*$, where $\Sigma={\rm diag}(\sigma_1(A),\sigma_2(A),\ldots,\sigma_{\rm min}(A))$. Consequently, the desired inequality can be proved by a similar way used in the proof of Corollary 9.6.7 in \cite{Be09}.
 \qed\end{proof}


\begin{proposition}\label{PthGLS}
Let $Q\in \mathbb{Q}^{m\times m}$ be positive definite. For any $U,V,W\in \mathbb{Q}^{m\times n}$, it holds that
$$
\|U-V\|_Q^2-\|W-V\|_Q^2-\|U-W\|_Q^2=2\langle U-W,Q(W-V)\rangle_R=2\langle Q(U-W),W-V\rangle_R.
$$
\end{proposition}
\begin{proof}
It can be proved by a similar way used in the proof of Pythagoras theorem of real matrices.
\qed\end{proof}

\begin{proposition}\label{Prop-4}
For any $A=(a_{ij}), B=(b_{ij})\in \mathbb{Q}^{m\times n}$, it holds that $|\langle A,B\rangle_R|\leq \|A\|_F\|B\|_F$.
\end{proposition}

\begin{proof}
Write $A=A_0+A_1\ii+A_2\jj+A_3\kk$ and $B=B_0+B_1\ii+B_2\jj+B_3\kk$, where $A_l,B_l\in \mathbb{R}^{m\times n}$ for $l=0,1,2,3$. It is obvious that $a_{ij}=(a_{ij})_0+(a_{ij})_1\ii+(a_{ij})_2\jj+(a_{ij})_3\kk$ and $b_{ij}=(b_{ij})_0+(b_{ij})_1\ii+(b_{ij})_2\jj+(b_{ij})_3\kk$ for any $i=1,2,\ldots,m$ and $j=1,2,\ldots,n$, where $(a_{ij})_l$ and $(b_{ij})_l$ are the $(i,j)$-th elements in $A_l$ and $B_l$, respectively, for $l=0,1,,2,3$. By the definition of quaternion matrices inner product, we have $\langle A,B\rangle+\langle B,A\rangle=\sum_{i=1}^m\sum_{j=1}^n(\bar b_{ij}a_{ij}+\bar a_{ij}b_{ij})$, which implies, together with Theorem 3 in \cite{QLY21}, that $\langle A,B\rangle_R=\frac{1}{2}\langle A,B\rangle+\langle B,A\rangle=\sum_{l=0}^3\langle A_l,B_l\rangle$. Consequently, we have
$$
|\langle A,B\rangle_R|\leq\displaystyle \sum_{l=0}^3\left|\langle A_l,B_l\rangle\right|\leq\displaystyle \sum_{l=0}^3\|A_l\|_F\|B_l\|_F\leq\displaystyle \sqrt{\sum_{l=0}^3\|A_l\|^2_F}\sqrt{\sum_{l=0}^3\|B_l\|^2_F}=\|A\|_F\|B\|_F,
$$
where the second and third inequality come Cauchy-Schwarz inequality for real matrices, and the last equality is due to the definition of Frobenius norm of quaternion matrices. We complete the proof.
\qed\end{proof}

\begin{theorem} \cite{Zh97} Any quaternion matrix $A\in \mathbb{Q}^{m\times n}$ has the following  QSVD form
$$
A=U\left[\begin{array}{cc}
\Sigma_r&O\\
O&O
\end{array}\right]V^*,
$$
 where $U\in\mathbb{Q}^{m\times m}$ and $V\in\mathbb{Q}^{n\times n}$  are unitary, and $\Sigma_r={\rm diag}(\sigma_1(A),\sigma_2(A),\ldots,\sigma_r(A))$ is a real  positive $r\times r$ diagonal matrix, with $\sigma_1(A)\geq \sigma_2(A)\geq\ldots\geq\sigma_r(A)>0$ as the nonzero singular values of $A$.
\end{theorem}

The rank of a quaternion matrix can be defined as the number of nonzero singular
 values, and the nuclear norm of a quaternion matrix $A$, denoted by $\|A\|_\circ$, is defined as the sum of all nonzero singular values, i.e., $\|A\|_\circ=\sum_{i=1}^r\sigma_i(A)$.

Let $f:\mathbb{Q}^{m\times n}\rightarrow \mathbb{R}$. We say that $f$ is a convex function if for any $X,X^\prime\in \mathbb{Q}^{m\times n}$ and any $t\in [0,1]$, we have $f(tX+(1-t)X^\prime)\leq tf(X)+(1-t)f(X^\prime)$.  Suppose that $f$ is differentiable at $X$ with respect to $X_i$ for $i=0,1,2,3$. We define the gradient of $f$ at $X$ as $\nabla f(X)=\frac{\partial f(X)}{\partial X_0}+\frac{\partial f(X)}{\partial X_1}\ii+\frac{\partial f(X)}{\partial X_2}\jj+\frac{\partial f(X)}{\partial X_3}\kk$. It is clear that $\nabla f(X)\in \mathbb{Q}^{m\times n}$. 
\begin{proposition}\label{Prop-AB3}\cite{CQZX20}
 Suppose that $f : \mathbb{Q}^{m\times n}\rightarrow \mathbb{R}$ is defined by $f(X)=\frac{1}{2}\|AX+B\|_F^2$, where $A\in \mathbb{Q}^{s\times m}$ and $B\in \mathbb{Q}^{s\times n}$. Then it holds that $\nabla f(X)=A^*(AX+B)$.
\end{proposition}

For given proper convex function $f:\mathbb{Q}^{m\times n}\rightarrow \mathbb{R}$ and $X\in \mathbb{Q}^{m\times n}$, the subdifferential of $f$ at $X$, denoted by $\partial f(X)$, is defined by
$$
\partial f(X)=\{Z\in \mathbb{Q}^{m\times n}:f(X^\prime)-f(X)-\langle Z,X^\prime- X\rangle_R\geq 0,~~\forall ~X^\prime\in \mathbb{Q}^{m\times n}\}.
$$
From Proposition 5.1 in \cite{QLWZ22}, the definition of $\langle \cdot,\cdot\rangle_R$ and the knowledge of convex  functions of real variables, we know that $\partial f(X)$ is a nonempty, convex and compact set in $\mathbb{Q}^{m\times n}$. The subdifferential $\partial f(X)$ is a singleton if and only if $f$ is differentiable at $X$. In this case, $\partial f( X)=\{\nabla f(X)\}$. Moreover, from the definition of $\partial f(\cdot)$, we know that, if $X^\diamond={\arg \min}_X f(X)$, then $0\in \partial f(X^\diamond)$.

 We present the following proposition, which shows that a subdifferential map of a convex quaternion matrix function is also monotone under $\langle\cdot,\cdot\rangle_R$, and will be used in later section.
\begin{proposition}\label{Monoto-Prop}
Let $f:\mathbb{Q}^{m\times n}\rightarrow \mathbb{R}$ be a proper convex function. For any $X,X^\prime\in {\rm dom} f$, it holds that
$$
\langle X-X^\prime, S-T\rangle_R\geq 0,~~~~\forall~S\in \partial f(X)~~{\rm and}~~T\in \partial f(X^\prime).
$$
\end{proposition}
\begin{proof}
It follows  immediately from Proposition 5.1 in \cite{QLWZ22}, the definition of $\langle \cdot,\cdot\rangle_R$ and the knowledge of convex  functions of real variables.
\qed\end{proof}


\subsection{Dual quaternion and dual quaternion matrices}\label{Sect2.3}
Denote by $\dot{\mathbb{Q}}$ the set of all dual quaternions. A dual quaternion $\dot{q}\in \dot{\mathbb{Q}}$ has the form $\dot{q} = q_{\rm st} + q_{\rm in}\epsilon$,
where $q_{\rm st}, q_{\rm in} \in \mathbb {Q}$ are the standard part and the infinitesimal part of $\dot{q}$, respectively. Similar to dual numbers, if $q_{\rm st} \not = 0$, then we say that $\dot{q}$ is appreciable; otherwise, we say that $\dot{q}$ is infinitesimal. 
For any $\dot{p}=p_{\rm st} + p_{\rm in}\epsilon, \dot{q}=q_{\rm st} + q_{\rm in}\epsilon\in \dot{\mathbb{Q}}$, denote $\dot{p}+\dot{q}=(p_{\rm st}+q_{\rm st})+(p_{\rm in}+q_{\rm in})\epsilon$, and the multiplication of $\dot{p}$ and $\dot{q}$ is defined as $\dot{p}\dot{q}=p_{\rm st}q_{\rm st} + (p_{\rm in}q_{\rm st}+p_{\rm st}q_{\rm in})\epsilon$. 
The conjugate of $\dot{q}$ is $\dot{q}^* = \bar q_{\rm st} + \bar q_{\rm in}\epsilon$, (see \cite{BK20,Ke12}).  
The magnitude of $\dot{q}=q_{\rm st}+q_{\rm in}\epsilon\in \dot{\mathbb{Q}}$ is defined as
\begin{equation} \label{e7}
\displaystyle|\dot{q}| := \left\{ \begin{array}{ll} |q_{\rm st}| +\displaystyle {(q_{\rm st}\bar q_{\rm in}+q_{\rm in} \bar q_{\rm st}) \over 2|q_{\rm st}|}\epsilon, & \ {\rm if}\  q_{\rm st} \not = 0, \\
|q_{\rm in}|\epsilon, &  \ {\rm otherwise},
\end{array} \right.
\end{equation}
which is a dual number, since $q_{\rm st}\bar q_{\rm in}+q_{\rm in} \bar q_{\rm st}\in \mathbb{R}$. 

Denote by $\dot{\mathbb{Q}}^{m\times n}$ the set of all $m\times n$ matrices with dual quaternion entries.
For given $U = U_{\rm st} + U_{\rm in} \epsilon = (u_{ij}), V=V_{\rm st} + V_{\rm in} \epsilon = (v_{ij}) \in {\dot{\mathbb Q}}^{m \times n}$, denote by $\langle U,V\rangle$ the dual quaternion-valued inner product, i.e., $\langle U,V\rangle=\sum_{i=1}^m\sum_{j=1}^nv_{ij}^*u_{ij}$, and the Frobenius norm of $U$, which is a dual number, is defined by
\begin{equation}\label{FNorm-DQM}
\|U \|_F = \left\{\begin{array}{ll}\displaystyle\sqrt{\sum_{i=1}^m \sum_{j=1}^n |u_{ij}|^2}, & \quad\ {\rm if}\ U_{\rm st} \not = O, \\
\|U_{\rm in}\|_F\epsilon,\quad & \quad \ {\rm otherwise.} \end{array}\right.
\end{equation}

\begin{proposition}\label{Prop-2}
Let $\dot A = A_{\rm st} + A_{\rm in} \epsilon \in {\dot{\mathbb Q}}^{m \times n}$ with $A_{\rm st}\neq O$. we have
$$
\|\dot A\|_F=\|A_{\rm st}\|_F+\frac{\langle A_{\rm st},A_{\rm in}\rangle_R}{\|A_{\rm st}\|_F}\epsilon.
$$
\end{proposition}

\begin{proof}
It follows from (\ref{FNorm-DQM}) and Theorem 3 in \cite{QLY21}.
\qed\end{proof}

\section{A dual-metric function on $\dot{\mathbb{Q}}^{m\times n}$ }\label{Metric Fun}
A function ${\bf\rho}: \dot{\mathbb{Q}}^{m\times n}\times \dot{\mathbb{Q}}^{m\times n}\rightarrow \dot{\mathbb{R}}$ is called a dual-metric on $\dot{\mathbb{Q}}^{m\times n}$, if it satisfies the following two conditions: i) ${\bf\rho}(\dot X,\dot Y)\geq 0$ for any $\dot X,\dot Y\in \dot{\mathbb{Q}}^{m\times n}$, and ${\bf\rho}(\dot X,\dot Y)=0$ if and only if $\dot X=\dot Y$; ii) ${\bf\rho}(\dot{X},\dot Y)\leq {\bf\rho}(\dot{X},\dot Z)+{\bf\rho}(\dot Z,\dot Y)$ for any $\dot{X},\dot Y,\dot Z\in \dot{\mathbb{Q}}^{m\times n}$. It is easy to verify that a metric function ${\bf\rho}$ defined on $\dot{\mathbb{Q}}^{m\times n}$  must satisfy the symmetry, i.e., ${\bf\rho}(\dot{X},\dot Y)={\bf\rho}(\dot Y,\dot{X})$ for any $\dot{X},\dot Y\in \dot{\mathbb{Q}}^{m\times n}$.

With the help of Frobenius norm $\|\cdot\|_F$ defined on $\mathbb{Q}^{m\times n}$, we present a dual metric function defined on $\dot{\mathbb{Q}}^{m\times n}$ as follows.

\begin{definition}\label{Def-Qnorm}
For given $\dot X=X_{\rm st}+X_{\rm in}\epsilon, \dot Y=Y_{\rm st}+Y_{\rm in}\epsilon\in \dot{\mathbb{Q}}^{m\times n}$, the dual metric of $\dot X$ and $\dot Y$, which is a dual number, is defined by
\begin{equation}\label{metric-f}
{\bf\rho}(\dot X,\dot Y)=\|X_{\rm st}-Y_{\rm st}\|_F+\|X_{\rm in}-Y_{\rm in}\|_F\epsilon.
\end{equation}
\end{definition}
Notice that, ${\bf\rho}(\cdot,\cdot)$ in \eqref{metric-f} is closely related to the function $v(\dot X):=\|X_{\rm st}\|_F+\|X_{\rm in}\|_F\epsilon$ where $\dot X=X_{\rm st}+X_{\rm in}\epsilon\in \dot{\mathbb{Q}}^{m\times n}$, but $v(\cdot)$ is not a norm function defined on $\dot{\mathbb{Q}}^{m\times n}$. In fact, it is easy to obtain the positivity and triangle inequality that the norm function must satisfy hold, due to the $\|X_{\rm st}\|_F$ and $\|X_{\rm in}\|_F$ are Frobenius norm of quaternion matrices $X_{\rm st}$ and $X_{\rm in}$, respectively. However, the homogeneity, i.e., $v(\dot q\dot X)= |\dot q|v(\dot X)$, does not hold, where $\dot q\in \dot{\mathbb{Q}}$, as shown in the following example.

\begin{example}
Let $\dot X=\left[\begin{array}{cc}
1&0\\
0&1
\end{array}\right]+\left[\begin{array}{cc}
2&0\\
0&3
\end{array}\right]\epsilon\in \widehat{\mathbb{R}}^{2\times 2}$ and $\dot q=2+3\epsilon\in\widehat{\mathbb{R}}$. It is easy to see that $v(\dot X)=\sqrt{2}+\sqrt{13}\epsilon$ and $|\dot q|=\dot q$, which implies $|\dot q|v(\dot X)=2\sqrt{2}+(2\sqrt{13}+3\sqrt{2})\epsilon$. Moreover, it is easy to see that $
\dot q\dot X=\left[\begin{array}{cc}
2&0\\
0&2
\end{array}\right]+\left[\begin{array}{cc}
7&0\\
0&9
\end{array}\right]\epsilon
$, which implies $v(\dot q\dot X)=2\sqrt{2}+\sqrt{130}\epsilon<|\dot q|v(\dot X)$.
\end{example}

\begin{proposition}
For given $\dot X=X_{\rm st}+X_{\rm in}\epsilon\in \dot{\mathbb{Q}}^{m\times n}$ and $\dot q=q_{\rm st}+q_{\rm in}\epsilon\in\dot{\mathbb{Q}}$, it holds that
$$
v(\dot q\dot X)\leq |\dot q|_Dv(\dot X),
$$
where $|\dot q|_D=|q_{\rm st}|+|q_{\rm in}|\epsilon\in \widehat{\mathbb{R}}_+$, which is a metric function defined on $\dot{\mathbb{Q}}$.
\end{proposition}

\begin{proof}
It is obvious that $\dot q\dot X=q_{\rm st}X_{\rm st}+(q_{\rm in}X_{\rm st}+q_{\rm st}X_{\rm in})\epsilon$. By the definition of $v(\cdot)$, we have
$v(\dot q\dot X)=\|q_{\rm st}X_{\rm st}\|_F+\|q_{\rm in}X_{\rm st}+q_{\rm st}X_{\rm in}\|_F\epsilon$. Since $\|q_{\rm st}X_{\rm st}\|_F=|q_{\rm st}|\|X_{\rm st}\|_F$ and $\|q_{\rm in}X_{\rm st}+q_{\rm st}X_{\rm in}\|_F\leq \|q_{\rm in}X_{\rm st}\|_F+\|q_{\rm st}X_{\rm in}\|_F=|q_{\rm in}|\|X_{\rm st}\|_F+|q_{\rm st}|\|X_{\rm in}\|_F$, by the definitions of the total order of dual numbers and $v(\cdot)$, we know
$$
v(\dot q\dot X)\leq |q_{\rm st}|\|X_{\rm st}\|_F+(|q_{\rm in}|\|X_{\rm st}\|_F+|q_{\rm st}|\|X_{\rm in}\|_F)\epsilon=|\dot q|_Dv(\dot X).
$$
We obtain the desired result and complete the proof.
\qed\end{proof}
It should be pointed that, for any $\dot q=q_{\rm st}+q_{\rm in}\epsilon\in \dot{\mathbb{Q}}$ with $q_{\rm st}=0$ or $q_{\rm in}=0$, we have $v(\dot q\dot X)=|\dot q|v(\dot X)$. In fact, when $q_{\rm in}=0$, $|\dot q|=|q_{\rm st}|$ and $\dot q\dot X=q_{\rm st}X_{\rm st}+q_{\rm st}X_{\rm in}\epsilon$, it is obvious that $v(\dot q\dot X)=\|q_{\rm st}X_{\rm st}\|_F+\|q_{\rm st}X_{\rm in}\|_F\epsilon=|q_{\rm st}|\|X_{\rm st}\|_F+|q_{\rm st}|\|X_{\rm in}\|_F\epsilon=|\dot q|v(\dot X)$; when $q_{\rm st}=0$, then $\dot q\dot X=q_{\rm in}X_{\rm st}\epsilon$, since $\epsilon^2=0$. By (\ref{e7}), it holds that $|\dot q|=|q_{\rm in}|\epsilon$. Moreover, by Definition \ref{Def-Qnorm}, we know
$$v(\dot q\dot X)=\|q_{\rm in}X_{\rm st}\|_{F}\epsilon=|q_{\rm in}|\|X_{\rm st}\|_{F}\epsilon=|q_{\rm in}|(\|X_{\rm st}\|_{F}+\|X_{\rm in}\|_{F}\epsilon)\epsilon=|\dot q|v(\dot X).$$


\begin{proposition}\label{Prop-3}
For any $\dot A = A_{\rm st} + A_{\rm in} \epsilon \in {\dot{\mathbb Q}}^{m \times n}$, it holds that $\|\dot A\|_F\leq v(\dot A)$.
\end{proposition}

\begin{proof}
If $A_{\rm st}=0$, it follows from the definitions of $\|\cdot\|_F$ and $v(\cdot)$. If $A_{\rm st}\neq0$, since $\langle A_{\rm st},A_{\rm in}\rangle_R\leq \|A_{\rm st}\|_F\|A_{\rm in}\|_F$ which from Proposition \ref{Prop-4}, by Proposition \ref{Prop-2}, we obtain the desired result and complete the proof.
\qed\end{proof}

\section{Least-squares problems of dual quaternion equation in the sense of metric function ${\bf\rho}$}\label{Sect4-0}
In this section, we study the algorithmic problem of the general approximate solution of the dual quaternion overdetermined equation and the sparse (or low-rank) approximate solution of this equation. Although the former is a special case of the latter, for the sake of completeness, we will start this section with the study of the general approximate solution of the dual quaternion overdetermined equation.

\subsection{Bi-level program formulation}\label{Sec4-1}
It is well-known that, due to various reasons, it may not be possible to obtain an exact solution for the equation
\begin{equation}\label{dqe}
\dot A\dot X=\dot B,
\end{equation}
 where $\dot A=A_{\rm st}+A_{\rm in}\epsilon\in \dot{\mathbb{Q}}^{m\times n}$ with $m>n$, $\dot B=B_{\rm st}+B_{\rm in}\epsilon\in \dot{\mathbb{Q}}^{m\times p}$ are given constant matrices, and $\dot X=X_{\rm in}+X_{\rm st}\epsilon\in \dot{\mathbb{Q}}^{n\times p}$ is unknown. Similar to real overdetermined equations, a natural approach for solving approximate solution of the dual quaternion overdetermined equation is to convert it into a least squares problem, in order to obtain an approximate solution in a certain sense. With the help of the metric function $\rho$ introduced in the previous section, in this subsection, we transform the original problem \eqref{dqe} into an unconstrained dual quaternion optimization problem in the following form
\begin{equation}\label{DQEOptim-1}
\begin{array}{cl}
\displaystyle\mathop{\rm min}\limits_{\dot X\in \dot{\mathbb{Q}}^{n\times p}}&\rho(\dot A\dot X,\dot B):=\|A_{\rm st}X_{\rm st}-B_{\rm st}\|_F+\|A_{\rm st}X_{\rm in}+A_{\rm in}X_{\rm st}-B_{\rm in}\|_F\epsilon.
\end{array}
\end{equation}
 Notice that, by the definition of total order for dual real numbers introduced in \cite{QLY21}, the model (\ref{DQEOptim-1}) is closely related to two quaternion optimization problems as follows
\begin{equation}\label{DQEOptim-Prob1}
\displaystyle\mathop{\rm min}\limits_{X_{\rm st}\in \mathbb{Q}^{n\times p}}~\displaystyle\frac{1}{2}\|A_{\rm st}X_{\rm st}-B_{\rm st}\|_F^2
\end{equation}
and
\begin{equation}\label{DQEOptim-Prob2}
\displaystyle\mathop{\rm min}\limits_{X_{\rm st},X_{\rm in}\in \mathbb{Q}^{n\times p}}~\displaystyle\frac{1}{2}\|A_{\rm st}X_{\rm in}+A_{\rm in}X_{\rm st}-B_{\rm in}\|_F^2.
\end{equation}
The models (\ref{DQEOptim-Prob1}) and (\ref{DQEOptim-Prob2}) are both typical quaternion least squares problems. Hence, we call (\ref{DQEOptim-1}) the least-squares problem of dual quaternion equation in the sense of the metric function $\rho$. According to the definitions of $\rho$ and the total order for dual real numbers, we need to find a quaternion matrix pair $(X_{\rm st}^\diamond,X_{\rm in}^\diamond)$, which ensures that $X_{\rm st}^\diamond$ is the global optimal solution of the model (\ref{DQEOptim-Prob1}), while also ensuring that $(X_{\rm st}^\diamond,X_{\rm in}^\diamond)$ must be the optimal solution of problem (\ref{DQEOptim-Prob2}). We must first obtain all solutions of the problem (\ref{DQEOptim-Prob1}) (generally speaking, in the case of ${\rm rank}(A_{\rm st})<n$, the problem (\ref{DQEOptim-Prob1}) has infinite many optimal solutions, which can be seen from the expression (\ref{Step1-X_st-solution}) given later), and then further obtain the optimal solution to the problem (\ref{DQEOptim-Prob2}) from all these optimal solutions already obtained. Denote by $\Xi$ the optimal solution set of problem (\ref{DQEOptim-Prob1}). Then, the second optimization problem (\ref{DQEOptim-Prob2}) can be expressed as the following form
\begin{equation}\label{DQEOptim-2}
\begin{array}{l}
\displaystyle\mathop{\rm min}\limits_{X_{\rm st}\in \Xi,~X_{\rm in}\in \mathbb{Q}^{n\times p}}~\displaystyle\frac{1}{2}\|A_{\rm st}X_{\rm in}+A_{\rm in}X_{\rm st}-B_{\rm in}\|_F^2.
\end{array}
\end{equation}
This solving process is very similar to solving a bi-level program with special structure. 

\subsection{Description of algorithms}\label{Sec4-2}
To facilitate the explanation of the numerical solution method, we first start with solving problems (\ref{DQEOptim-Prob1}) and (\ref{DQEOptim-Prob2}). Generally speaking, it is a very difficult task to obtain a global solution of bi-level programming while ensuring the exact solution of the lower level programming (which corresponds to the standard part of the metric function $\rho(AX,B)$). Fortunately, thanks to the special structure of (\ref{DQEOptim-Prob1}) and Theorem 3.2.1 (also see Theorem 3.3.3) in \cite{WLZZ18}, we know that its optimal solution set can be expressed as follow
\begin{equation}\label{Step1-X_st-solution}
\Xi=\big\{X_{\rm st}=A_{\rm st}^\dag B_{\rm st}+(I_n-A_{\rm st}^\dag A_{\rm st})Z~:~\forall~Z\in \mathbb{Q}^{n\times p}\big\},
\end{equation}
where $A_{\rm st}^\dag$ denotes the Moore-Penrose inverse of the quaternion matrix $A_{\rm st}$, see Definition 1.6.1 in \cite{WLZZ18}.

Consequently, by substituting (\ref{Step1-X_st-solution}) into the objective function of (\ref{DQEOptim-Prob2}), the model (\ref{DQEOptim-Prob2})  can be rewritten as
\begin{equation}\label{DQEOptim-3}
\displaystyle\mathop{\rm min}\limits_{X_{\rm in},Z\in \mathbb{Q}^{n\times p}}~\displaystyle \frac{1}{2}\|A_{\rm st}X_{\rm in}+CZ+D\|_F^2,
\end{equation}
where $C=A_{\rm in}(I_n-A_{\rm st}^\dag A_{\rm st})\in \mathbb{Q}^{m\times n}$ and $D=A_{\rm in}A_{\rm st}^\dag B_{\rm st}-B_{\rm in}\in \mathbb{Q}^{m\times p}$, which is a quaternion version of quadratic programming with coupled variables. Write $G=[A_{\rm st},C]\in \mathbb{Q}^{m\times 2n}$ and $W=[X_{\rm in};Z]\in \mathbb{Q}^{2n\times p}$. The optimization problem (\ref{DQEOptim-3}) can be further expressed as
\begin{equation}\label{DQEOptim-4}
\displaystyle\mathop{\rm min}\limits_{W\in \mathbb{Q}^{2n\times p}}~\displaystyle \frac{1}{2}\|GW+D\|_F^2.
\end{equation}
By Theorem 3.2.1 in \cite{WLZZ18} again, the optimal solution of (\ref{DQEOptim-4}) can be expressed as follows
\begin{equation}\label{Step1-W_st-solution}
W^\diamond=-G^\dag D+(I_{2n}-G^\dag G)U,~~~~\forall~U\in \mathbb{Q}^{2n\times p},
\end{equation}
where $G^\dag$ denotes the Moore-Penrose inverse of the quaternion matrix $G$.

Considering the high computational cost required for calculating $G^\dag$, we use the iterative schemes of proximity operator to solve  (\ref{DQEOptim-4}). We assume ${\rm rank}(G)=r$, which implies that there exists an orthogonal matrix $V\in \mathbb{Q}^{2n\times 2n}$ such that $G=[G_1,O]V$, where $G_1\in \mathbb{Q}^{m\times r}$ is full column rank. Denote $W_1=V_1W$ and $W_2=V_2W$, where $V_1\in \mathbb{Q}^{r\times 2n}$ and $V_2\in \mathbb{Q}^{(2n-r)\times 2n}$ are matrices composed of the first $r$ row vectors and the last $2n-r$ row vectors in $V$, respectively. Then (\ref{DQEOptim-4}) can be rewritten as
\begin{equation}\label{DQEOptim-5}
\displaystyle\mathop{\rm min}\limits_{W_1\in \mathbb{Q}^{r\times p}}~\displaystyle f_0(W_1):=\frac{1}{2}\|G_1W_1+D\|_F^2.
\end{equation}

We apply proximal point minimization algorithm to solve \eqref{DQEOptim-5}. For the obtained $k$-th iterate $W_1^k$, we compute the next iterate $W_1^{k+1}$ via solving
\begin{equation}\label{k-Iterative}
\displaystyle W_1^{k+1}=\mathop{\rm argmin}\limits_{W_1\in \mathbb{Q}^{r\times p}}\;f_0(W_1)+\frac{\tau}{2} \big\|W_1-W_1^k\big\|_F^2,
\end{equation}
where $\tau>0$ is a proximity parameter. By the optimality condition of (\ref{k-Iterative}) and Proposition \ref{Prop-3}, it holds that $
G_1^*\big(G_1W_1^{k+1}+D\big)+\tau\big(W_1^{k+1}-W_1^k\big)=0$, which implies
\begin{equation}\label{X-subprob-k-Solu}
\displaystyle W_1^{k+1}=(\tau I_r+G_1^*G_1)^{-1}\big\{\tau W_1^k-G_1^*D\big\}.
\end{equation}
The complete iterative process is listed in Algorithm \ref{PrOalg-1}.

\begin{algorithm}[!htb]
	\caption{(Implementable proximal minimization algorithm for \eqref{DQEOptim-5}).}\label{PrOalg-1}
	\begin{algorithmic}[1]
		\STATE {\bf Input:} $\tau>0$,  $\varepsilon>0$ and starting point $W_1^0\in \mathbb{Q}^{r\times p}$.
				\FOR{$k=0,1,2,\cdots$}
		\STATE Compute $W_1^{k+1}$ via \eqref{X-subprob-k-Solu}.
        \STATE If $\big\|W_1^{k+1}-W_1^{k}\big\|_F\leq \varepsilon$ is satisfied, {\bf output}: Optimal solution $W_1^\diamond=W_1^{k+1}$.
		\ENDFOR
	\end{algorithmic}
\end{algorithm}
After obtaining $W_1^\diamond\in \mathbb{Q}^{r\times p}$, take any $W_2^\diamond\in \mathbb{Q}^{(2n-r)\times p}$, and let $W^\diamond=V^*[W_1^\diamond;W_2^\diamond]$. Let $X_{\rm in}^\diamond=W^\diamond(1:n;:)$ and $X_{\rm st}^\diamond=A_{\rm st}^\dag B_{\rm st}+(I_n-A_{\rm st}^\dag A_{\rm st})Z^\diamond$ with $Z^\diamond=W^\diamond(n+1:2n;:)$. Through this process, we obtain an optimal solution $(X_{\rm st}^\diamond,X_{\rm in}^\diamond)$ of (\ref{DQEOptim-2}), which means, together with the fact $X_{\rm st}^\diamond\in \Xi^\diamond$, that $\dot X^\diamond=X_{\rm st}^\diamond+X_{\rm in}^\diamond\epsilon$ is an optimal solution of (\ref{DQEOptim-1}).

Now, we study the algorithmic problem of approximate solutions for dual quaternion overdetermined equations under certain regularity conditions (such as low-rankness, sparsity).  Following the above basic idea for handing with the dual quaternion least squares problem, a dual quaternion matrix least squares problem with an infinitesimal matrix regularization term can be similarly transformed into the optimization problems as follows
\begin{equation}\label{DQEOptim-S2}
\left\{
\begin{array}{l}
\displaystyle\mathop{\rm min}\limits_{X_{\rm st}\in \mathbb{Q}^{n\times p}}~\displaystyle\frac{1}{2}\|A_{\rm st}X_{\rm st}-B_{\rm st}\|_F^2,\\
\displaystyle\mathop{\rm min}\limits_{X_{\rm st},X_{\rm in}\in \mathbb{Q}^{n\times p}}~\displaystyle\frac{1}{2}\|A_{\rm st}X_{\rm in}+A_{\rm in}X_{\rm st}-B_{\rm in}\|_F^2+\alpha \varphi(X_{\rm in}),
\end{array}
\right.
\end{equation}
where $\varphi(X_{\rm in})$ is a regularization term used to characterize the feature of $X_{\rm in}$, and $\alpha>0$ is a parameter.

 Similar to solving (\ref{DQEOptim-1}), by the optimal solution set expression $\Xi$ of the first problem in (\ref{DQEOptim-S2}), the second optimization problem in (\ref{DQEOptim-S2}) can be converted into
 \begin{equation}\label{DQEOptim-31}
\displaystyle\mathop{\rm min}\limits_{X_{\rm in},Z\in \mathbb{Q}^{n\times p}}~\displaystyle g(X_{\rm in},Z):=\frac{1}{2}\|A_{\rm st}X_{\rm in}+CZ+D\|_F^2+\alpha \varphi(X_{\rm in}),
\end{equation}
where $C$ and $D$ are same to ones in (\ref{DQEOptim-3}). Moreover, the problem (\ref{DQEOptim-31}) is equivalent to
 \begin{equation}\label{DQEOptim-3AC}
 \begin{array}{cl}
\displaystyle\mathop{\rm min}\limits_{X_{\rm in}\in \mathbb{Q}^{n\times p},W\in \mathbb{Q}^{2n\times p}}&\displaystyle \frac{1}{2}\|GW+D\|_F^2+\alpha \varphi(X_{\rm in})\\
{\rm s.t.}&X_{\rm in}=HW,
\end{array}
\end{equation}
where $G$ and $W$ are same to ones in (\ref{DQEOptim-4}), and $H=[I_n,O]\in \mathbb{Q}^{n\times 2n}$.

Considering that $X_{\rm in}$ often exhibits low-rankness in many practical problems, we focus on the algorithmic problem of (\ref{DQEOptim-3AC}) with $\varphi(X_{\rm in})={\rm rank}(X_{\rm in})$. For general form of $\varphi(X_{\rm in})$,  as long as the relevant $X_{\rm in}$-subproblem (see below) has closed-form solution, our algorithm design idea is still feasible.
To efficiently minimize (\ref{DQEOptim-3AC}) with rank function ${\rm rank}(\cdot)$, by utilizing the nuclear norm of quaternion matrices, 
 we relax (\ref{DQEOptim-3AC}) into
\begin{equation}\label{DQEOptim-3ACR}
 \begin{array}{cl}
\displaystyle\mathop{\rm min}\limits_{X_{\rm in}\in \mathbb{Q}^{n\times p},W\in \mathbb{Q}^{2n\times p}}&\displaystyle \frac{1}{2}\|GW+D\|_F^2+\alpha\|X_{\rm in}\|_\circ\\
{\rm s.t.}&X_{\rm in}=HW,
\end{array}
\end{equation}
where $\|X_{\rm in}\|_\circ$ denotes the nuclear norm of quaternion matrix $X_{\rm in}$, which can approximately characterize the low-rankness of $X_{\rm in}$. Inspired by the method in \cite{DLPY17}, we proposed a quaternion matrix version of Jacobi-Proximal ADMM algorithm to solve (\ref{DQEOptim-3ACR}). Write $f(X_{\rm in})=\alpha \|X_{\rm in}\|_\circ$ and $g(W)=\frac{1}{2}\|GW+D\|_F^2$.
Denote
$$
\mathcal{L}(X_{\rm in},W,T)= f(X_{\rm in})+g(W)-\langle X_{\rm in}-HW,T\rangle_R+\frac{\rho}{2}\|X_{\rm in}-HW\|_F^2.
$$
For given $k$-th iterate $(X_{\rm in}^k,W^k,T^k)$, we first update $X_{\rm in}$ and $W$ in parallel as follows:
\begin{equation}\label{XW-k-th}
\left\{
\begin{array}{l}
\displaystyle X_{\rm in}^{k+1}=\mathop{\arg\min}\limits_{X_{\rm in}}~f(X_{\rm in})+\frac{\rho}{2}\|X_{\rm in}-HW^k-(1/\rho)T^k\|_F^2+\frac{\tau_X}{2}\|X_{\rm in}-X_{\rm in}^k\|_F^2,\\
\displaystyle W^{k+1}=\mathop{\arg\min}\limits_{W}~g(W)+\frac{\rho}{2}\|X_{\rm in}^k-HW-(1/\rho)T^k\|_F^2+\frac{\tau_W}{2}\|W-W^k\|_F^2.
\end{array}
\right.
\end{equation}
It is obvious that solving the first problem in (\ref{XW-k-th}), i.e., $X_{\rm in}$-subproblem, is equivalent to solving
\begin{equation}\label{X-subProb}
\displaystyle X_{\rm in}^{k+1}=\mathop{\arg\min}\limits_{X_{\rm in}}~\alpha\|X_{\rm in}\|_\circ+\frac{\tau_X+\rho}{2}\left\|X_{\rm in}-\tilde X^k\right\|_F^2,
\end{equation}
where $\tilde X^k=\frac{\rho HW^k+T^k+\tau_X X^k_{\rm in}}{\rho+\tau_X}\in \mathbb{Q}^{n\times p}$. By Theorem 2 (The quaternion matrix singular value thresholding operator theorem) in \cite{MKL20}, we know that the optimal solution of (\ref{X-subProb}) is given by
\begin{equation}\label{X-subProb-Solution}
 X_{\rm in}^{k+1}=U_r\Sigma_r\Big(\frac{\alpha}{\rho+\tau_X}\Big) V_r^*,
\end{equation}
where $U_r=[{\bf u}_1,\ldots,{\bf u}_r]\in \mathbb{Q}^{n\times r}, V_r=[{\bf v}_1,\ldots,{\bf v}_r]\in \mathbb{Q}^{p\times r}$ come from the compact QSVD $\tilde X^k=U_r\Sigma_rV_r^*$ with $\Sigma_r={\rm diag}(\sigma_1,\ldots,\sigma_r)$, and $\Sigma_r(\delta)={\rm daig}({\rm max}(\sigma_1-\delta,0),\ldots, {\rm max}(\sigma_r-\delta,0))$ for $\delta>0$.

Moreover, since $W^{k+1}$ is the optimal solution of $W$-subproblem, by Proposition \ref{Prop-AB3}, it holds that
$$
G^*(GW^{k+1}+D)+\rho H^*\big(HW^{k+1}-X_{\rm in}^k+(1/\rho)T^k\big)+\tau_W\big(W^{k+1}-W^k\big)=0,
$$
which implies
\begin{equation}\label{Wk-solution}
W^{k+1}=\big(\tau_W I_{2n}+G^*G+\rho H^* H\big)^{-1}\big\{H^*(\rho X_{\rm in}^k-T^k)+\tau_W W^k-G^* D\big\}.
\end{equation}

After obtaining $(X_{\rm in}^{k+1}, W^{k+1})$, we then update the Lagrange multiplier $T$ as follows:
\begin{equation}\label{Update-T-k}
T^{k+1}=T^k-\gamma \rho(X_{\rm in}^{k+1}-HW^{k+1}),
\end{equation}
where $\gamma>0$.

We summarize the updating schemes for (\ref{DQEOptim-3ACR}) in Algorithm \ref{alg-2}.
\begin{algorithm}[!htb]
	\caption{(Implementable Jacobi-Proximal ADMM algorithm for \eqref{DQEOptim-3ACR}).}\label{alg-2}
	\begin{algorithmic}[1]
		\STATE {\bf Input:} $\tau_X>0$, $\tau_W>0$, $\rho>0$, $\gamma>0$, $\varepsilon>0$ and starting point $(X_{\rm in}^0,W^0,T^0)$.
				\FOR{$k=0,1,2,\cdots$}
		\STATE Compute $(X_{\rm in}^{k+1},W^{k+1})$ via \eqref{X-subProb-Solution} and (\ref{Wk-solution}).
		\STATE Update $T^{k+1}$ via \eqref{Update-T-k}.
        \STATE If $\big\|(X_{\rm in}^{k+1},W^{k+1})-(X_{\rm in}^{k},W^{k})\big\|_F\leq \varepsilon$ is satisfied, {\bf output}: Optimal solution $(X_{\rm in}^\diamond,W^\diamond)=(X_{\rm in}^{k+1},W^{k+1})$.
		\ENDFOR
	\end{algorithmic}
\end{algorithm}

After obtaining  $(X_{\rm in}^\diamond,W^\diamond)$, it is easy to see that $X^\diamond=X_{\rm st}^\diamond+X_{\rm in}^\diamond \epsilon$ with $X_{\rm st}^\diamond=A_{\rm st}^\dag B_{\rm st}+(I_n-A_{\rm st}^\dag A_{\rm st})W^\diamond(n+1:2n;:)$ is a desired least squares solution of dual quaternion equations under low-rank regularity condition.

\section{Convergence analysis}\label{Conanalysis}
In this section, we investigate the convergence of Algorithm \ref{PrOalg-1} and Algorithm \ref{alg-2}.
\subsection{Convergence  of Algorithm \ref{PrOalg-1}}
Let $W_1^\diamond\in \mathbb{Q}^{r\times p}$. We say that $W_1^\diamond$ is a stationary point of (\ref{DQEOptim-5}), if $G^*\big(GW_1^\diamond+D\big)=0$. Notice that, since $f_0$ is a strictly convex function on $\mathbb{Q}^{r\times p}$, which is due to the fact that $G_1$ is full column rank,  $W_1^\diamond$ is the unique global optimal solution of the problem (\ref{DQEOptim-5}). Although the proof of the convergence property of Algorithm \ref{PrOalg-1} is common, for the sake of completeness, we still provide its proof here.
\begin{proposition}\label{Conve-Prop1}
Let $\big\{W_1^{k}\big\}$ be the sequence generated by Algorithm \ref{PrOalg-1}. For any positive integer $k$, we have
\begin{equation}\label{Conve-Prop1-e01}
f_0(W_1^{k+1})+\frac{\tau}{2} \big\|W_1^{k+1}-W_1^k\big\|_F^2\leq f_0(W_1^k).
\end{equation}
\end{proposition}

\begin{proof}
It follows immediately from the fact that $W_1^{k+1}$ is the optimal solution of (\ref{k-Iterative}) at the $k$-th update.
\qed\end{proof}

From (\ref{Conve-Prop1-e01}), we know that $\big\{f_0(W_1^k)\big\}_{k=0}^\infty$ is nonincreasing, i.e., $f_0(W_1^{k+1})\leq f_0(W_1^k)$ for any positive integer $k$. Consequently, from $f_0(W_1^{k})\leq f_0(W_1^0)$, we know $\|G_1W^k_1\|_F\leq\|G_1W^k_1+D\|_F+\|D\|_F\leq \sqrt{2f_0(W_1^0)}+\|D\|_F$. By Proposition \ref{QM-Ineq}, it holds that
$$
\sigma_{\rm min}(G_1)\|W^k_1\|_F\leq \|G_1W^k_1\|_F\leq \sqrt{2f_0(W_1^0)}+\|D\|_F,
$$
which implies that $\big\{W_1^k\big\}_{k=0}^\infty$ is bounded, since $\sigma_{\rm min}(G_1)>0$ which from the assumption $G_1$ is full column rank.

\begin{theorem}\label{Golbal-ConTh-1}
Let $\big\{W_1^k\big\}_{k=0}^\infty$ be a sequence generated by
Algorithm \ref{PrOalg-1}. We have

(a) $\mathop{\rm lim}\limits_{k\rightarrow \infty}\big\|W_1^{k+1}-W_1^k\big\|_F=0$.

(b) any cluster point $W_1^\diamond$ of $\big\{W_1^k\big\}_{k=0}^\infty$ is a stationary point of (\ref{DQEOptim-5}).
\end{theorem}
\begin{proof}
We first prove statement (a). By Proposition \ref{Conve-Prop1}, we know that $\big\{f_0(W_1^k)\big\}$ is nonincreasing, which implies, together with the fact $f_0(W_1^k)\geq 0$ for any positive integer $k$, that $\lim_{k\rightarrow\infty}f_0(W_1^k)=f_0^\diamond$ exists. By Proposition \ref{Conve-Prop1} again, for any positive integer $N$, we have
\begin{equation}\label{Conve-Prop1-e4}
\frac{\tau}{2}\sum_{k=0}^N  \big\|W_1^{k+1}-W_1^k\big\|_F^2\leq f_0(W_1^0)-f_0(W_1^{N+1})\leq f_0(W_1^0),
\end{equation}
which implies, by letting $N\rightarrow +\infty$, that
\begin{equation}\label{Conve-Prop1-e5}
\sum_{k=0}^\infty\big\|W_1^{k+1}-W_1^k\big\|_F^2<+\infty,
\end{equation}
which implies that the statement (a) holds.

We now prove statement (b). Since the sequence $\big\{W_1^k\big\}_{k=0}^\infty$ is bounded, we know that a cluster point of $\big\{W_1^k\big\}_{k=0}^\infty$ exists. Suppose that $W_1^\diamond$  is a cluster point of the sequence $\big\{W_1^k\big\}_{k=0}^\infty$ and let $\big\{W_1^{k_i}\big\}_{i=1}^\infty$ be a convergent subsequence such that $
\lim_{i\rightarrow \infty}W_1^{k_i}=W_1^\diamond$. Since $W_1^{k_i}$ is the optimal solution of subproblem (\ref{k-Iterative}) with $k=k_i-1$, from its first-order optimality condition, we know
\begin{equation}\label{Optimcondition-Sub0}
G_1^*\big(G_1W_1^{k_i}+D\big)+\tau\big(W_1^{k_i}-W_1^{k_i-1}\big)=0.
\end{equation}
for  any $i=1,2,\ldots$. By letting $i\rightarrow \infty$, we know, together with (a), that $G_1^*\big(G_1W_1^\diamond+D\big)=0$, which means that $W_1^\diamond$ is a stationary point of (\ref{DQEOptim-5}). We complete the proof.
\qed\end{proof}

\subsection{Convergence of Algorithm \ref{alg-2}}
Let $(X^\diamond_{\rm in},W^\diamond, T^\diamond)\in \mathbb{H}:=\mathbb{Q}^{n\times p}\times\mathbb{Q}^{2n\times p}\times \mathbb{Q}^{n\times p}$. We say that $(X^\diamond_{\rm in},W^\diamond,T^\diamond)$ satisfying
\begin{equation}\label{OrgPKKT}
\left\{
\begin{array}{l}
T^\diamond\in \partial f(X^\diamond_{\rm in}),~-H^*T^\diamond=\nabla g(W^\diamond),
\\
X^\diamond_{\rm in}-HW^\diamond=0
\end{array}
\right.
\end{equation}
is a KKT pair of (\ref{DQEOptim-3ACR}). Here, $T^\diamond$ is celled a Lagrange multiplier associated with $(X^\diamond_{\rm in},W^\diamond)$. 
Through of this paper, we assume that the KKT pair set $\Xi$ of (\ref{DQEOptim-3ACR}) is nonempty, i.e.,
$$
\Xi:=\{(X^\diamond_{\rm in},W^\diamond, T^\diamond)\in \mathbb{H}~:~T^\diamond\in \partial f(X^\diamond_{\rm in}),-H^*T^\diamond=\nabla g(W^\diamond), X^\diamond_{\rm in}-HW^\diamond=0\}\neq\emptyset.$$
For given parameters $\tau_X, \tau_W, \rho, \gamma>0$, denote $Q={\rm diag}\big((\tau_X+\rho)I_n,\tau_W I_{2n}+\rho H^*H, (1/\gamma \rho)I_n\big)$ and
$$
P=\left[\begin{array}{ccc}
(\tau_X+\rho)I_n&0&\displaystyle\frac{1}{\gamma}I_n\\
0&\tau_W I_{2n}+\rho H^*H&\displaystyle\frac{1}{\gamma}H^*\\
\displaystyle\frac{1}{\gamma}I_n&\displaystyle\frac{1}{\gamma}H &\displaystyle\frac{2-\gamma}{\gamma^2\rho}I_n
\end{array}
\right].
$$
It is easy to see that $Q$ is positive definite and $P$ is Hermitian. Moreover, we have

\begin{proposition}\label{P-Positive-Definite} Suppose that the positive parameters $\tau_X, \tau_W, \gamma,\rho$ satisfy $\gamma(\tau_X+\rho)>1$, $\gamma\tau_W>1$ and $2-\gamma-2\gamma \rho>0$. Then, the quaternion Hermitian matrix $P$ is positive definite.
\end{proposition}
\begin{proof}
 For any $u=(x^\top, y^\top,z^\top)^\top\in (\mathbb{Q}^{n}\times\mathbb{Q}^{2n}\times \mathbb{Q}^{n})\backslash\{0\}$, we have
$$
\begin{array}{lll}
u^*Pu&=&\displaystyle(\tau_X+\rho)\|x\|^2+\tau_W \|y\|^2+\rho \|Hy\|^2+\frac{2-\gamma}{\gamma^2\rho}\|z\|^2\\
&&\displaystyle+\frac{1}{\gamma}(\langle z,x \rangle+\langle x,z \rangle)+\frac{1}{\gamma}(\langle Hy,z \rangle+\langle z, Hy \rangle)\\
&\geq&\displaystyle(\tau_X+\rho)\|x\|^2+\tau_W \|y\|^2+\rho \|Hy\|^2+\frac{2-\gamma}{\gamma^2\rho}\|z\|^2-\frac{2}{\gamma}\|x\|\|z\|-\frac{2}{\gamma}\|Hy\|\|z\|\\
&\geq&\displaystyle(\tau_X+\rho)\|x\|^2+\tau_W \|y\|^2+\frac{2-\gamma}{\gamma^2\rho}\|z\|^2-\frac{2}{\gamma}\|x\|\|z\|-\frac{2}{\gamma}\|y\|\|z\|\\
&\geq&\displaystyle(\tau_X+\rho)\|x\|^2+\tau_W \|y\|^2+\frac{2-\gamma}{\gamma^2\rho}\|z\|^2\displaystyle-\frac{1}{\gamma}(\|x\|^2+\|z\|^2)-\frac{1}{\gamma}(\|y\|^2+\|z\|^2)\\
&=&\displaystyle\frac{\gamma(\tau_X+\rho)-1}{\gamma}\|x\|^2+\frac{\tau_W\gamma-1}{\gamma} \|y\|^2+\frac{2-\gamma-2\gamma \rho}{\gamma^2\rho}\|z\|^2\\
&>&0,
\end{array}
$$
where the first inequality is due to Proposition \ref{Prop-4}, and the second inequality comes from the fact $\|Hy\|\leq \|y\|$ since $H=[I_n,O]$, which means that $P$ is positive definite.
\qed\end{proof}
\begin{lemma}\label{MPMonoto}
Let $\big\{\Theta^k:=\big(X^k_{\rm in}, W^k, T^k\big)\big\}_{k=0}^\infty$ be the sequence generated by Algorithm \ref{alg-2} from any initial point. Then, for any $\Theta^\diamond:=\big(X^\diamond_{\rm in}, W^\diamond, T^\diamond\big)\in \Xi$ and $k\geq 0$, we have
\begin{equation}\label{Lemma-1-e1}
\|\Theta^k-\Theta^\diamond\|_Q^2-\|\Theta^{k+1}-\Theta^\diamond\|_Q^2\geq \|\Theta^k-\Theta^{k+1}\|_P^2.
\end{equation}
\end{lemma}

\begin{proof} Since $X^{k+1}_{\rm in}$ and $W^{k+1}$ are the optimal solutions of $X_{\rm in}$- and $W$-subproblems in (\ref{XW-k-th}), respectively, by Proposition \ref{Prop-AB3}, we have
 \begin{equation}\label{KKT-Condition}
 \left\{
 \begin{array}{ll}
 \tau_X\big(X^{k}_{\rm in}-X_{\rm in}^{k+1}\big)+\hat T^k+\rho H\big(W^k-W^{k+1}\big)\in\partial f(X^{k+1}_{\rm in})\\
 \tau_W\big(W^{k}-W^{k+1}\big)-H^*\hat T^k+\rho H^*\big(X_{\rm in}^k-X_{\rm in}^{k+1}\big)= \nabla g(W^{k+1}),
\end{array}
 \right.
 \end{equation}
 where $\hat T^k=T^k-\rho \big(X^{k+1}-HW^{k+1}\big)$. Consequently, since $\Theta^\diamond\in \Xi$, by Proposition \ref{Monoto-Prop}, we have
 $$
 \left\{
 \begin{array}{l}
 \big\langle X^{k+1}_{\rm in}-X^\diamond_{\rm in}, \tau_X\big(X^{k}_{\rm in}-X_{\rm in}^{k+1}\big)+\hat T^k+\rho H\big(W^k-W^{k+1}\big)- T^\diamond\big\rangle_R\geq 0,\\
 \big\langle W^{k+1}- W^\diamond, \tau_W\big(W^{k}-W^{k+1}\big)-H^*\hat T^k+\rho H^*\big(X_{\rm in}^k-X_{\rm in}^{k+1}\big)+H^* T^\diamond\big\rangle_R\geq 0,\\
  \end{array}
 \right.
 $$
 which implies, by summing the above inequality, that
 \begin{equation}\label{Con-Lem-e1}
 \begin{array}{l}
 \big\langle X^{k+1}_{\rm in}-X^\diamond_{\rm in}, \tau_X\big(X^{k}_{\rm in}-X_{\rm in}^{k+1}\big)+\hat T^k+\rho H\big(W^k-W^{k+1}\big)- T^\diamond\big\rangle_R\\
 + \big\langle W^{k+1}-W^\diamond, \tau_W\big(W^{k}-W^{k+1}\big)-H^*\hat T^k+\rho H^*\big(X_{\rm in}^k-X_{\rm in}^{k+1}\big)+H^*T^\diamond\big\rangle_R\geq 0.
 \end{array}
  \end{equation}
  It is obvious that
 $$
  \begin{array}{l}
  \big\langle X^{k+1}_{\rm in}-X^\diamond_{\rm in}, \tau_X\big(X^{k}_{\rm in}-X_{\rm in}^{k+1}\big)\big\rangle_R+\big\langle W^{k+1}-W^\diamond, \tau_W\big(W^{k}-W^{k+1}\big)\big\rangle_R\\
  =
 \big\langle M^{k+1}- M^\diamond, {\rm diag}(\tau_X I,\tau_W I)\big(M^{k}-M^{k+1}\big)\big\rangle_R
  \end{array}
  $$
and
$$
\big\langle X^{k+1}_{\rm in}-X^\diamond_{\rm in}, \hat T^k-T^\diamond\big\rangle_R\\
 -\big\langle W^{k+1}-W^\diamond, H^*\big(\hat T^k-T^\diamond\big)\big\rangle_R=\big\langle N\big(M^{k+1}-M^\diamond\big), \hat T^k-T^\diamond\big\rangle_R,
 $$
 where $N:=[I_n,-H]$ and $M^\diamond:=[X^\diamond_{\rm in};W^\diamond]$. Since $X_{\rm in}-HW=NM$ for any $M:=[X_{\rm in};W]$, we have $H(W^k-W^{k+1})=(X_{\rm in}^k-X_{\rm in}^{k+1})-N(M^k-M^{k+1})$ and $H(W^{k+1}-W^\diamond)=(X_{\rm in}^{k+1}-X^\diamond_{\rm in})-N(M^{k+1}-M^\diamond)$. Consequently, it holds that
$$
\begin{array}{l}
\big\langle X^{k+1}_{\rm in}-X^\diamond_{\rm in}, H\big(W^k-W^{k+1}\big)\big\rangle_R+\big\langle W^{k+1}-W^\diamond,H^*\big(X_{\rm in}^k-X_{\rm in}^{k+1}\big)\big\rangle_R\\
=\big\langle X^{k+1}_{\rm in}-X^\diamond_{\rm in}, X^k-X^{k+1}\big\rangle_R-\big\langle X^{k+1}_{\rm in}-X^\diamond_{\rm in}, N\big(M^k-M^{k+1}\big)\big\rangle_R\\
~~+\big\langle H\big(W^{k+1}-W^\diamond\big),H\big(W^k-W^{k+1}\big)\big\rangle_R+\big\langle H\big(W^{k+1}-W^\diamond\big),N\big(M^k-M^{k+1}\big)\big\rangle_R.
\end{array}
$$
Hence (\ref{Con-Lem-e1}) can be rewritten as
 \begin{equation}\label{Con-Lem-e2}
 \begin{array}{l}
 \big\langle M^{k+1}-M^\diamond, {\rm diag}(\tau_X I,\tau_W I)\big(M^{k}-M^{k+1}\big)\big\rangle_R+\big\langle N\big(M^{k+1}-M^\diamond\big), \hat T^k-T^\diamond\big\rangle_R\\
 +\rho\big\langle X^{k+1}_{\rm in}-X^\diamond_{\rm in}, X^k-X^{k+1}\big\rangle_R+\rho\big\langle H\big(W^{k+1}-W^\diamond\big),H\big(W^k-W^{k+1}\big)\big\rangle_R\\
 \geq \rho\big\langle N\big(M^{k+1}-M^\diamond\big),N\big(M^k-M^{k+1}\big)\big\rangle_R.
 \end{array}
  \end{equation}
 Note that $N(M^{k+1}-M^\diamond)=\frac{1}{\gamma \rho}\big(T^k-T^{k+1}\big)$ from $N M^\diamond=0$ and
 $$
 \hat T^k-T^\diamond=\big(\hat T^k-T^{k+1}\big)+(T^{k+1}-T^\diamond)=\frac{\gamma-1}{\gamma}\big(T^k-T^{k+1}\big)+(T^{k+1}-T^\diamond).
 $$
 By (\ref{Con-Lem-e2}), we know
$$
 \begin{array}{l}
\displaystyle\frac{1}{\gamma \rho}\big\langle T^{k}-T^{k+1}, \hat T^{k+1}-T^\diamond\big\rangle_R +\big\langle M^{k+1}-M^\diamond, {\rm diag}(\tau_X I,\tau_W I)\big(M^{k}-M^{k+1}\big)\big\rangle_R\\
 +\big\langle X^{k+1}_{\rm in}-X^\diamond_{\rm in},\rho I\big(X^{k}_{\rm in}-X^{k+1}_{\rm in}\big)\big\rangle_R+\big\langle W^{k+1}-W^\diamond,\rho H^*H\big(W^k-W^{k+1}\big)\big\rangle_R\\
 \geq \displaystyle\frac{1-\gamma}{\gamma^2\rho}\|T^k-T^{k+1}\|_F^2+\frac{1}{\gamma}\big\langle T^k-T^{k+1},N\big(M^k-M^{k+1}\big)\big\rangle_R.
 \end{array}
  $$
 or more compactly,
\begin{equation}\label{Con-Lem-e3}
 \begin{array}{l}
\displaystyle\big\langle \Theta^{k+1}-\Theta^\diamond, Q\big(\Theta^{k}-\Theta^{k+1}\big)\big\rangle_R
 \geq \displaystyle\frac{1-\gamma}{\gamma^2\rho}\|T^k-T^{k+1}\|_F^2+\frac{1}{\gamma}\big\langle T^k-T^{k+1},N\big(M^k-M^{k+1}\big)\big\rangle_R.
 \end{array}
  \end{equation}
 Since $\|\Theta^k-\Theta^\diamond\|_Q^2-\|\Theta^{k+1}-\Theta^\diamond\|_Q^2=2\big\langle \Theta^{k+1}-\Theta^\diamond, Q\big(\Theta^{k}-\Theta^{k+1}\big)\big\rangle_R+\|\Theta^k-\Theta^{k+1}\|_Q^2$ from Proposition \ref{PthGLS}, using  the above inequality (\ref{Con-Lem-e3}) yields (\ref{Lemma-1-e1}) immediately.
\qed\end{proof}

If we choose the positive parameters $\tau_X, \tau_W, \gamma,\rho$ satisfying $\gamma(\tau_X+\rho)>1$, $\gamma\tau_W>1$ and $2-\gamma-2\gamma \rho>0$, then by Proposition \ref{P-Positive-Definite}, we know that the quaternion Hermitian matrix $P$ is positive definite, which implies there exists $\kappa>0$ such that $\|\Theta^k-\Theta^{k+1}\|_P^2\geq \kappa\|\Theta^k-\Theta^{k+1}\|_F^2$. Consequently, by Lemma \ref{MPMonoto}, we have $\|\Theta^k-\Theta^\diamond\|_Q^2-\|\Theta^{k+1}-\Theta^\diamond\|_Q^2\geq \kappa\|\Theta^k-\Theta^{k+1}\|_F^2$, which implies the iterative sequence $\{\Theta^k\}_{k=0}^\infty$ is Fej\'{e}r monotone with respect to $\Xi$. See Definition 5.1 in \cite{BC10}.

\begin{theorem}
Let $\big\{\Theta^k\big\}_{k=0}^\infty$ be the sequence generated by Algorithm \ref{alg-2} from any initial point. If the positive parameters $\tau_X, \tau_W, \gamma,\rho$ satisfy $\gamma(\tau_X+\rho)>1$, $\gamma\tau_W>1$ and $2-\gamma-2\gamma \rho>0$, then the sequence $\{\Theta^k\}_{k=0}^\infty$ converges a KKT pair $\Theta^\star$ of (\ref{DQEOptim-3ACR}), i.e., $\|\Theta^k-\Theta^\star\|_F\rightarrow 0$ for some $\Theta^\star\in \Xi$.
\end{theorem}

\begin{proof}
For given parameters $\tau_X, \tau_W, \gamma,\rho$, it is obvious that $Q$ and $P$ are positive definite, which implies, together with Lemma \ref{MPMonoto}, that the error $\|\Theta^k-\Theta^\diamond\|_Q^2$ is monotonically non-increasing and thus converging, as well as $\|\Theta^k-\Theta^{k+1}\|_F^2\rightarrow 0$.

Moreover, by Lemma \ref{MPMonoto}, we know $\|\Theta^k-\Theta^\diamond\|_Q^2\leq \|\Theta^{k-1}-\Theta^\diamond\|_Q^2\leq \ldots\leq \|\Theta^0-\Theta^\diamond\|_Q^2$, which means that $\{\Theta^k\}_{k=0}^\infty$ is bounded, and hence a cluster point of $\{\Theta^k\}_{k=0}^\infty$ exists. Let $\Theta^\star$ be a cluster point of the sequence $\{\Theta^k\}_{k=0}^\infty$, and let $\{\Theta^{k_i}\}_{k_i=1}^\infty$ be a convergent subsequence such that $\lim_{i\rightarrow\infty} \Theta^{k_i}=\Theta^\star$. Since $X^{k_i+1}_{\rm in}$ and $W^{k_i+1}$ are the optimal solution of the related $X_{\rm in}$- and $W$-subproblems with $k=k_i$, respectively, by Proposition \ref{Prop-AB3}, it holds that
\begin{equation}\label{ki-KKT-Condition}
 \left\{
 \begin{array}{ll}
 \tau_X\big(X^{k_i}_{\rm in}-X_{\rm in}^{k_i+1}\big)+\hat T^{k_i}+\rho H\big(W^{k_i}-W^{k_i+1}\big)\in\partial f(X^{k_i+1}_{\rm in})\\
 \tau_W\big(W^{k_i}-W^{k_i+1}\big)-H^*\hat T^{k_i}+\rho H^*\big(X_{\rm in}^{k_i}-X_{\rm in}^{k_i+1}\big)= \nabla g(W^{k_i+1}).
\end{array}
 \right.
 \end{equation}
Since $\|\Theta^{k_i}-\Theta^{k_i+1}\|_F\rightarrow 0$ which implies $\|T^{k_i}-T^{k_i+1}\|_F\rightarrow 0$, from (\ref{Update-T-k}), we know that $X^{k_i+1}-HW^{k_i+1}\rightarrow 0$ as $k_i\rightarrow\infty$, and hence $\hat T^{k_i}\rightarrow T^\star$ from the definition of $\hat T^{k_i}$. Consequently, from (\ref{ki-KKT-Condition}) and $\|\Theta^{k_i}-\Theta^{k_i+1}\|_F\rightarrow 0$, we know $T^\star\in \partial f(X^{\star}_{\rm in})$ and $-H^*\hat T^{\star}= \nabla g(W^{\star})$ as well as $X^{\star}-HW^{\star}=0$, which means $\Theta^{\star}\in \Xi$. Consequently, by Lemma \ref{MPMonoto} again, we know $\|\Theta^k-\Theta^\star\|_Q^2\geq \|\Theta^{k+1}-\Theta^\star\|_Q^2$, which implies that $\|\Theta^k-\Theta^\star\|_Q\rightarrow 0$, or equivalently, $\|\Theta^k-\Theta^\star\|_F\rightarrow 0$. We complete the proof.
\qed\end{proof}

\section{Numerical experiments}\label{NumTest}
		
In this section, we aim to conduct the performance of two proposed algorithms on synthetic data and color images. For brevity, we denote Algorithms \ref{PrOalg-1} and \ref{alg-2} as Alg1 and Alg2, respectively. For both synthetic data and color images scenarios, we consider three different cases: (i) $m<n$ and ${\rm rank(A_{\rm st})} = m $; (ii) $m>n$ and ${\rm rank}(A_{\rm st})=n$; (iii) $m>n$ and ${\rm rank}(A_{\rm st})<n$. Obviously, cases (ii) means that the problem (\ref{DQEOptim-1}) has a unique closed-form solution, whereas in cases (i) and (iii), the problem (\ref{DQEOptim-1}) has infinitely many solutions as ${\rm rank}(A_{\rm st})$ is not full column rank, see (\ref{Step1-X_st-solution}) and (\ref{Step1-W_st-solution}). So cases (i) and (iii) are more difficult to solve than (ii). For the case (iii), we first produce a quaternion matrix $A_{\rm st}^{\rm temp}\in\mathbb{Q}^{m\times n}$ with ${\rm rank}(A_{\rm st}^{\rm temp})=n$, and then keep only its first ${\rm round}(n/1.2)$ (In Matlab) singular values, resulting in a matrix $A_{\rm st}$ with no full column rank.  Furthermore, due to the wide applications of the low-rank properties, we here consider $f(\cdot)$ in problem (\ref{DQEOptim-S2}) as $\alpha\|\cdot\|_*$ in the following numerical experiments, where $\alpha> 0$.

Throughout this section, we set the $\varepsilon = 10^{-9}$ in two algorithms and the maximum iteration is 200. For the parameters emerged in Alg1 and Alg2, we set $\tau=1$ for Alg1, and $\alpha = 10^{-7}$, $\tau_X=1, \tau_W=0.6, \rho=2.5,\gamma=0.5$ for Alg2. 
And we define
\begin{align}\label{RSE-define}
	{\rm Objst} = ~\displaystyle\|A_{\rm st}X_{\rm st}^\diamond-B_{\rm st}\|_F~~~~~{\rm and}~~~~~
	{\rm Objin} = ~\displaystyle\|A_{\rm st}X_{\rm in}^\diamond+A_{\rm in}X_{\rm st}^\diamond-B_{\rm in}\|_F
\end{align}
to measure the quality of the solution $X_{\rm st}^\diamond$ and $X_{\rm in}^\diamond$.
All experiments were conducted on a laptop computer with Inter (R) core (TM) i7-7500 CPU @ 2.70GHz and 8HG memory.	

\subsection{Synthetic data}
Firstly, we compare Alg1 to the existing work \cite{WCW23}  (denote WCW's solution) in solving dual matrix least square problem, i.e., the sepecial case of (6) when $A,X,B$ are dual matrix. We generate $A_{\rm st},A_{\rm in}\in\mathbb{R}^{m\times n}$ and $X_{\rm st},X_{\rm in}\in\mathbb{R}^{n\times p}$ with the entries being random samples drawn from a Gaussian distribution, and we simplely take $n=200$ and $p=1$. Then, we generate $B_{\rm st} = A_{\rm st}X_{\rm st}$ and $B_{\rm in} = A_{\rm st}X_{\rm in}+A_{\rm in}X_{\rm st}$, respectively. The numerical results including Objst, Objin, Iteration (Iter for short) and computing time in seconds (Time for short) are summarized in Table \ref{compared-alg}. It is not difficult to see that Alg1 and WCW's solution can achieve satisfied results for case (i) and (ii). However, in case (iii), our Alg1 outperforms WCW's solution greatly in terms of Objin. This is because when $A_{\rm st}$ is not column full rank, two free variables $Z$ and $U$ will emerge, as shown in (\ref{Step1-X_st-solution}) and (\ref{Step1-W_st-solution}). In these cases, the different free variable $Z$ will impact the quality of $X_{\rm in}$. WCW's solution does not take into account the optimal case of the free variable $Z$, so the quality of Objin is relative low in scenario (iii). Therefore, the reliability of the Alg1 in solving different scenarios of $A_{\rm st}$ can be demonstrated.

\begin{table}
	\centering
	\caption{Computational results of Alg1 and WCW method on synthetic data}\label{compared-alg}
	\small\begin{tabular}{c c c c c c c c c c c c c c c c c}\toprule
		\multirow{2}{*}{case}&\multicolumn{1}{c}{\multirow{2}{*}{$(m,n)$}}&\multicolumn{4}{c}{Alg1}&&\multicolumn{4}{c}{WCW's solution}&\\
		\cmidrule{3-7}\cmidrule{8-12}
		&&Objst&Objin&Iter&Time&&Objst&Objin&Iter&Time  \\
		\hline
		
		\multirow{2}{*}{(i)}&\multicolumn{1}{c}{\multirow{1}{*}{$(150,200)$}}&4.7e-13&8.8e-13&10 &0.030 &&3.6e-13&5.1e-13&/ & 0.0031\\	
		&\multicolumn{1}{c}{\multirow{1}{*}{$(180,200)$}}  &7.3e-13&4.6e-12&18 &0.031 &&5.2e-13&7.4e-13&/ & 0.0049\\
		\hline
	
		\multirow{2}{*}{(ii)}&\multicolumn{1}{c}{\multirow{1}{*}{$(250,200)$}}	&6.3e-13&6.3e-10&12 &0.033&&5.0e-13& 8.5e-13&/ &0.0048\\	
		&\multicolumn{1}{c}{\multirow{1}{*}{$(300,200)$}}  &9.2e-13&2.1e-11&10 &0.024  &&7.7e-13&1.0e-12&/ & 0.0055\\
		\hline
	
		\multirow{2}{*}{(iii)}&\multicolumn{1}{c}{\multirow{1}{*}{$(250,200)$}}  &7.4e-13&1.4e-09&10 &0.026 &&5.5e-13&5.6e+01&/ &0.0042\\
		&\multicolumn{1}{c}{\multirow{1}{*}{$(300,200)$}}  &8.6e-13&9.0e-13&10 & 0.028 &&7.0e-13&8.3e+01&/ &0.0049 \\
		\toprule
	\end{tabular}
\end{table}

%
%
%

\begin{table}
\centering
	\caption{Computational results of two algorithms on synthetic data}\label{synthetic-low-rank}
	\small\begin{tabular}{cc c c c c c c c c c c c c c c c}\toprule
		\multirow{2}{*}{case}&\multicolumn{1}{c}{\multirow{2}{*}{$m$}}&\multicolumn{4}{c}{Alg1}&&\multicolumn{4}{c}{Alg2}&\\
		\cmidrule{3-7}\cmidrule{8-12}
		&&Objst&Objin&Iter&Time&&Objst&Objin&Iter&Time  \\
		\hline
		
		\multirow{2}{*}{(i)}&\multicolumn{1}{c}{\multirow{1}{*}{$150$}}	&5e-10&1e-09&10&1.62 &&4e-10&2e-07&39 &8.35\\
		&\multicolumn{1}{c}{\multirow{1}{*}{$180$}}  &1e-09&4e-07&10&1.76 &&9e-10&9e-07&40 &9.08 \\
		\hline
		
		\multirow{2}{*}{(ii)}&\multicolumn{1}{c}{\multirow{1}{*}{$250$}}	&1e-09&2e-10&10  &1.65 &&1e-09&2e-08&42 &9.43 \\
		&\multicolumn{1}{c}{\multirow{1}{*}{$300$}}  &2e-09	&2e-10&10 &1.41 &&2e-09&1e-08&39 &9.25\\
		\hline
		
		\multirow{2}{*}{(iii)}&\multicolumn{1}{c}{\multirow{1}{*}{$250$}}  &1e-09&1e-05&12 &2.05 &&1e-09&2e-06&41 &9.22\\
		&\multicolumn{1}{c}{\multirow{1}{*}{$300$}}  &2e-09&9e-08&10 &1.54 &&1e-09&2e-07&39 &9.26\\
		\toprule
	\end{tabular}

\end{table}	

\begin{figure}
	\centering
	\includegraphics[width=0.30\textwidth]{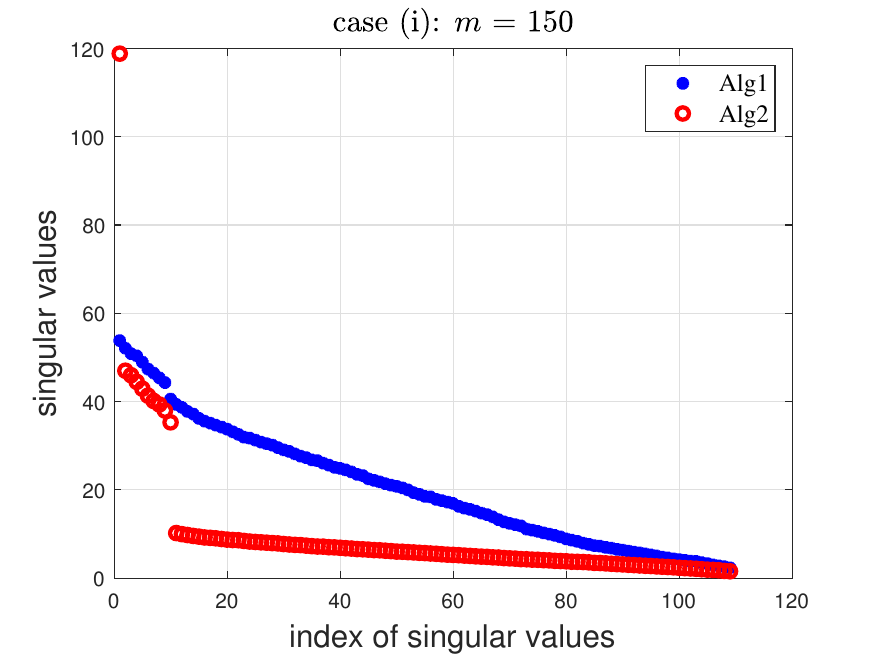}
	\includegraphics[width=0.30\textwidth]{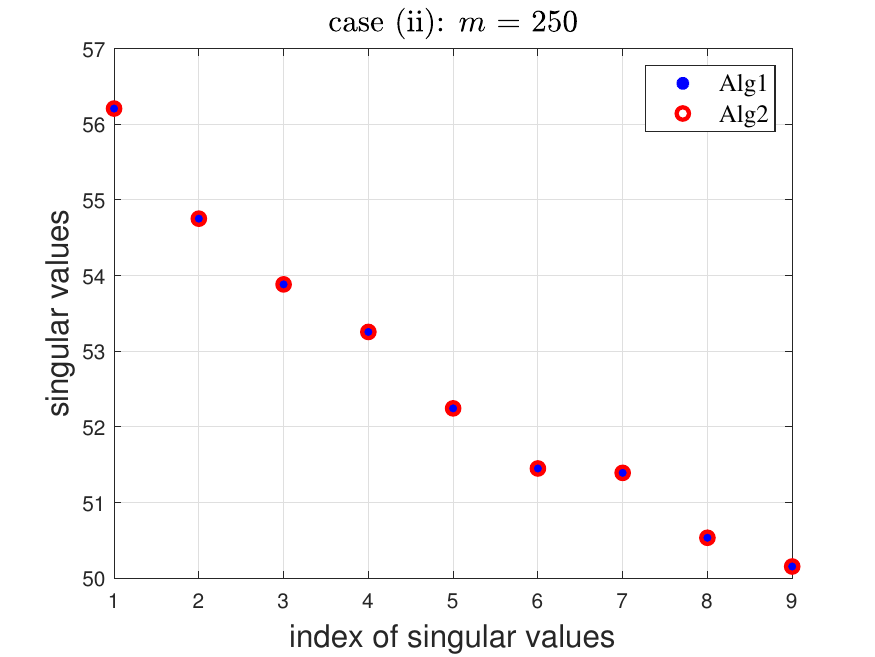}
	\includegraphics[width=0.30\textwidth]{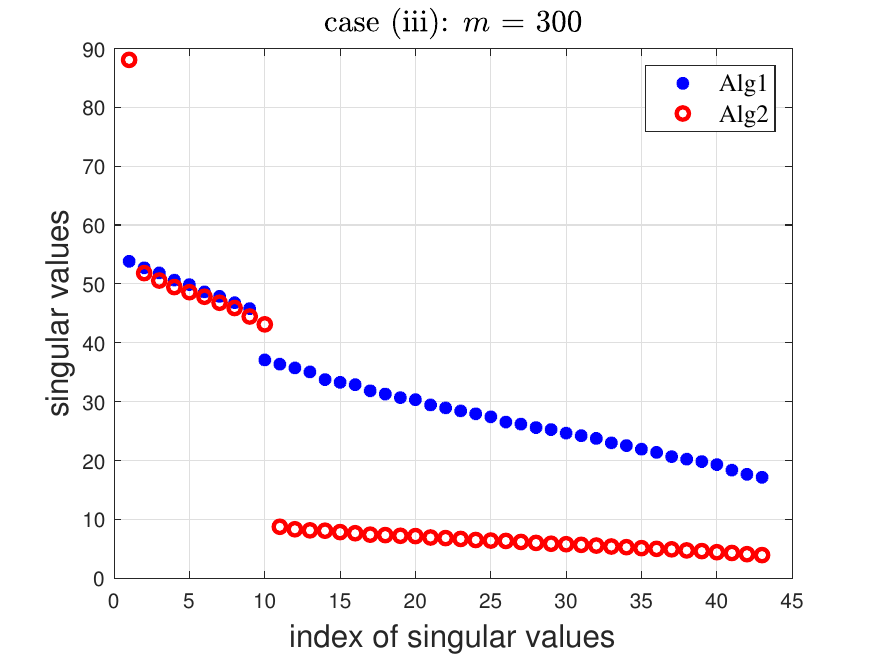}\\
	\caption{Results of singular values obtained by Alg1 and Alg2 in the synthetic scenario}\label{synthetic-test-SVD}
\end{figure}

Secondly, we conduct Alg1 and Alg2 in some complex scenarios, i.e., the problem of the low-rank approximate solutions
of the dual quaternion overdetermined equations. We randomly generate $A_{\rm st}, A_{\rm in}\in\mathbb{Q}^{m\times 200}$ and $X_{\rm st}\in\mathbb{Q}^{200\times 200}$  with the normal distribution, where $A_{\rm in}$ has only $5\%$ of non-zero elements.  In the part, we choose $m \in \{150,180,250,300\}$. For generating $X_{\rm in}$, we randomly generate a full-rank quaternion matrix $X_{\rm temp}\in\mathbb{Q}^{200\times 200}$, then keep only its maximum 10 singular values of $X_{\rm temp}$, and finally produce a quaternion matrix $X_{\rm in}$ of rank 10. The numerical results for the low-rank case including Objst, Objin, Iter and Time are summarized in Table \ref{synthetic-low-rank}, where shows that all cases can be effectively solved by Alg1 and Alg2. For detailed, Alg1 and Al2 have similar result in terms of Objst, and Alg2 has lower Objin, Iter and Time results than Alg1. However, as can be seen from Fig. \ref{synthetic-test-SVD}, Xin obtained by Alg2 has more tight singular value distribution due to the regularization term $\|\cdot\|_*$, and only focuses on the first 10 singular values.

\subsection{Color images}

Finally, we are concerned with the numerical performance of our approaches on real-world datasets. We here consider four widely used color images including `house', `sailboat' `peppers' and `baboon' (see Fig. \ref{images-test}), which are all size of $256\times 256\times 3$. The color image $\mathcal{X}\in\mathbb{R}^{n\times p\times 3}$ can be reshaped as a pure quaternion matrix $X\in\mathbb{Q}^{n\times p}$ by using the way $X=\mathcal{X}(:,:,1)\ii+\mathcal{X}(:,:,2)\jj+\mathcal{X}(:,:,3)\kk$. In this subsection, we set $X_{\rm st}, X_{\rm in}$ as color images and consider the following inverse color images problem. In detailed, we have the given quaternion matrices, $A_{\rm st}, A_{\rm in} \in\mathbb{Q}^{m\times n}$ and $ B_{\rm st}, B_{\rm in}\in\mathbb{Q}^{n\times p}$, where $A_{\rm st}$ and $A_{\rm in}$ are encryption matrices (i.e., secret key), and $B_{\rm st}$ and $B_{\rm in}$ (see Fig. \ref{images-test-2} for example) are regarded as the sent information by encrypting the real image information, i.e., $X_{\rm st}$ and $X_{\rm in}$, through the encryption matrices $A_{\rm st}$ and $A_{\rm in}$. Our goal is to estimate a satisfied $X_{\rm st}^\diamond$ and a $X_{\rm in}^\diamond$ with the given information $A_{\rm st}, A_{\rm in}, B_{\rm st}$ and $B_{\rm in}$, which could be reduced to the problem (\ref{DQEOptim-1}).

For specific experimental settings, the way to generate $A_{\rm st},A_{\rm in}\in\mathbb{Q}^{m\times 256}$ and $B_{\rm st}, B_{\rm in}\in\mathbb{Q}^{256\times 256}$ is same to synthetic scenario. We choose $m\in\{180,250,300,400\}$ and test two examples: (i) sailboat ($X_{\rm st}$) + house $(X_{\rm in} $) and (ii) peppers ($X_{\rm st}$) + baboon $(X_{\rm in} $).
All results obtained by Alg1 and Alg2 are summarized in Table \ref{color-low-rank}. We find that Alg1 outperforms Alg2 in terms of Objin, Iter and Time. Due to the visuality of images, we in addition plot some visual results in Fig. \ref{visual-image1} and Fig. \ref{visual-image2}. Form these two figures, it is clearly that the $X_{\rm in}^\diamond$ obtained by Alg2 is visually better than solved by Alg1, especially in cases (i) and (iii). The reason is simply that in both cases, $X_{\rm st}^\diamond$ and $X_{\rm in}^\diamond$ have infinitely many solutions, and perhaps the solution found by Alg1 can achieve a lower results in terms of Objst and Objin, but because Alg2 takes into account the low-rank property of images, the $X_{\rm st}^\diamond$ and $X_{\rm in}^\diamond$ obtained by it will be more reliable than Alg1. Furthermore, the low-rankness of $X_{\rm in}$ obtained by Alg2 are plotted in Fig. \ref{images-test-SVD}.


\begin{figure}
	\centering
	\includegraphics[width=0.90\textwidth]{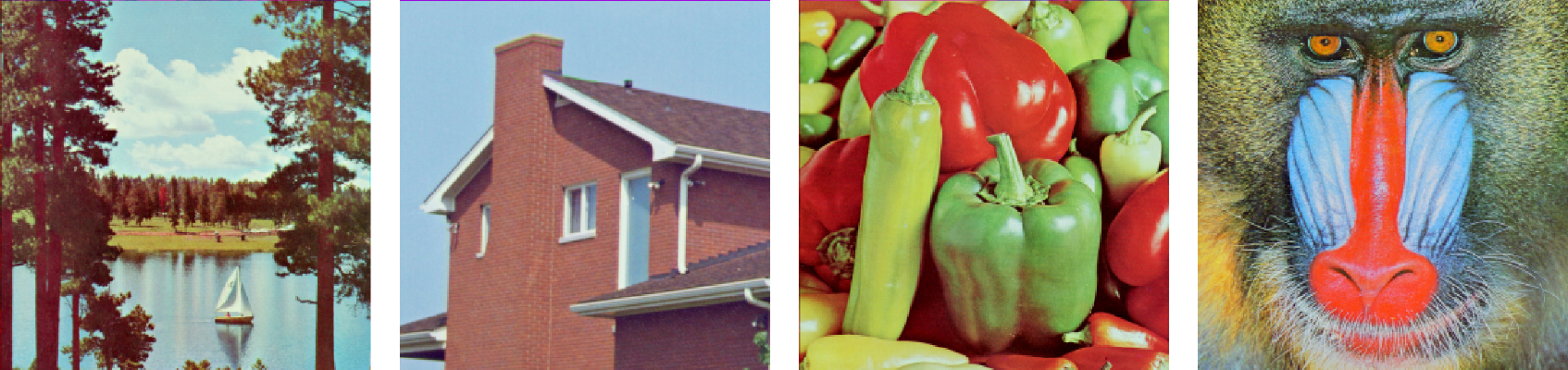}
	\caption{Four test images. From left to right: sailboat, house, peppers and baboon}\label{images-test}
\end{figure}

\begin{figure}
	\centering
	\includegraphics[width=0.9\textwidth]{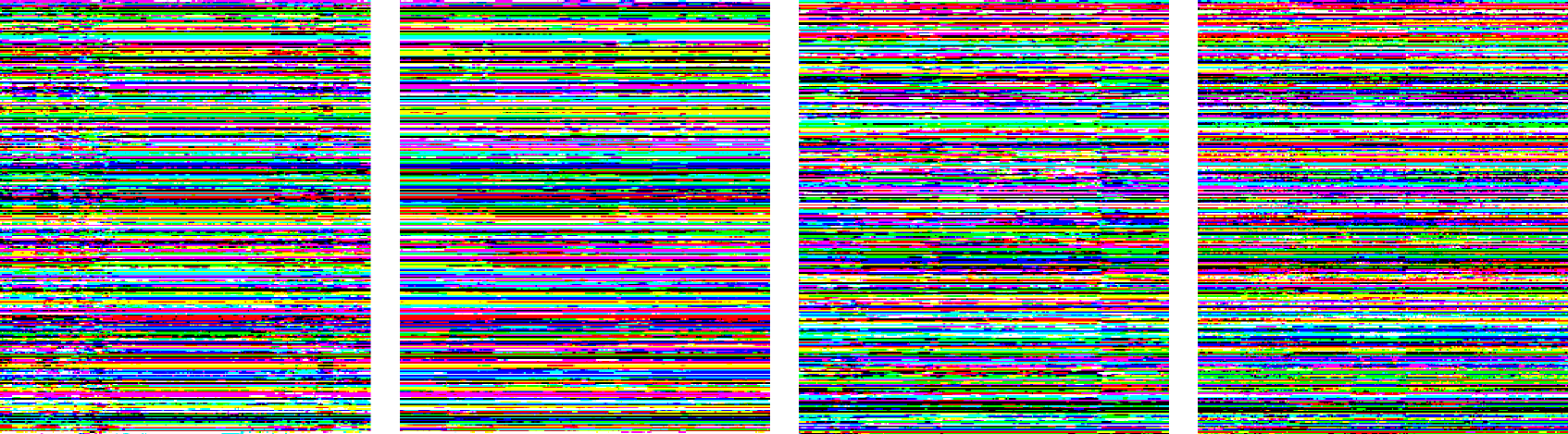}
	\caption{Visual results of $B_{\rm st}$ and $B_{\rm in}$ with ``sailboat + house'' (first two) and ``peppers + baboon'' (last two) in case (ii) and $m=300$.}\label{images-test-2}
\end{figure}

\begin{table}
\centering
	\caption{Computational results of two algorithms on inverse color images problem}\label{color-low-rank}
	\begin{center}
	{\begin{tabular}{c c c c c c c c c c c c c c c c c c c c}\toprule
	 \multirow{2}{*}{case}&\multicolumn{1}{c}{\multirow{2}{*}{$m$}}&\multicolumn{2}{c}{\multirow{2}{*}{Method}}&\multicolumn{4}{c}{sailboat~+~house}&&\multicolumn{4}{c}{peppers + baboon}\\
	\cmidrule{5-8}\cmidrule{9-13}
	&\multicolumn{2}{c}{}&&Objst&Objin&Iter&Time&&Objst&Objin&Iter &Time  \\
	\hline
		\multirow{4}{*}{(i)}&\multicolumn{1}{c}{\multirow{2}{*}{$180$}}&	\multicolumn{2}{c}{Alg1}&7e-10&4e-09&10&2.41 &&1e-09&4e-09&10&2.43 \\
		&&\multicolumn{2}{c}{Alg2}&3e-10&6e-08&43&15.12 &&5e-10&6e-08&43&15.46   \\		
		&\multicolumn{1}{c}{\multirow{2}{*}{$250$}}&	\multicolumn{2}{c}{Alg1}&1e-09&7e-07&37&9.97 &&8e-10&5e-07&33&8.71 \\
		&&\multicolumn{2}{c}{Alg2}&3e-10&2e-07&109&43.43&&3e-10&2e-07&93&37.35 \\
		\hline
		
		\multirow{4}{*}{(ii)}&\multicolumn{1}{c}{\multirow{2}{*}{$300$}}&	\multicolumn{2}{c}{Alg1}&1e-09&7e-09&10&3.05 &&1e-09&5e-09&10&3.00 \\
		&&\multicolumn{2}{c}{Alg2}&4e-10&1e-07&45&18.50 &&6e-10&1e-07&45&18.54 \\		
		&\multicolumn{1}{c}{\multirow{2}{*}{$400$}}&\multicolumn{2}{c}{Alg1}&2e-09&2e-10&11&3.46 &&2e-09&2e-10&10&3.44\\
		&&\multicolumn{2}{c}{Alg2}&6e-10&6e-08&44&19.68 &&7e-10&7e-08&44&19.12 \\
		\hline
		
		\multirow{4}{*}{(iii)}&\multicolumn{1}{c}{\multirow{2}{*}{$300$}}&	\multicolumn{2}{c}{Alg1}&1e-09&4e-07&10&3.85 &&2e-09&3e-07&18&5.28  \\
		&&\multicolumn{2}{c}{Alg2}&4e-10&8e-07&44&18.37 &&9e-10&8e-08&43&18.18 \\		
		&\multicolumn{1}{c}{\multirow{2}{*}{$400$}}&	\multicolumn{2}{c}{Alg1}&2e-09&5e-08&10&3.44 &&1e-09&4e-08&10&3.49\\
		&&\multicolumn{2}{c}{Alg2}&6e-10&7e-08&44 &19.94 &&5e-10&8e-08&44 &19.98 \\
		\toprule
	\end{tabular}}
	\end{center}
\end{table}

\begin{figure}
	\centering
	\includegraphics[width=0.90\textwidth]{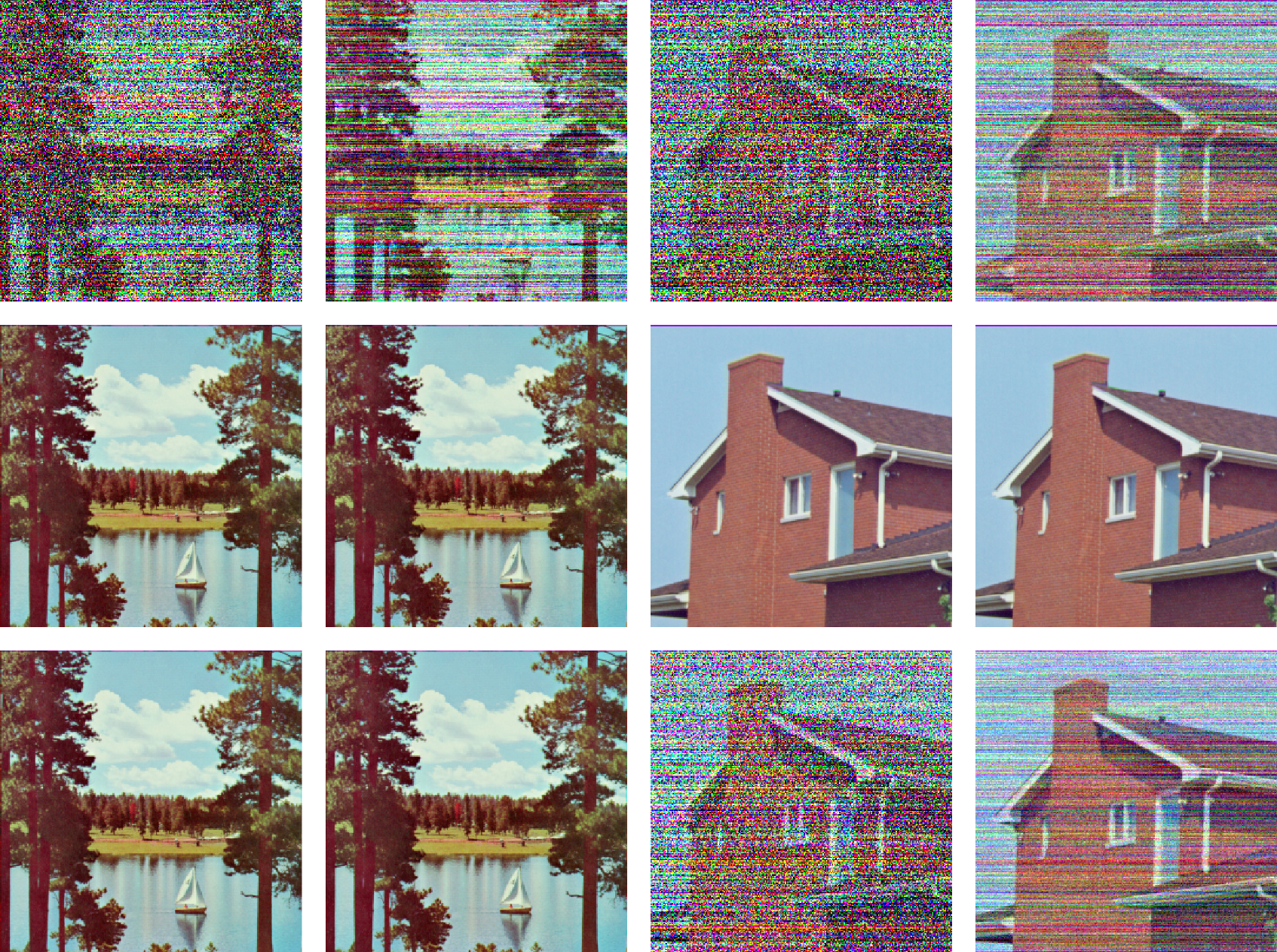}
	\caption{Visual results of $X_{\rm st}$ and $X_{\rm in}$ in ``sailboat + house''. \\ First row: case (i) with $m=180$; Second row: case (ii) with $m=300$; Third row: case (ii) with $m=400$. The first two columns are $X_{\rm st}^\diamond$ obtained by Alg1 and Alg2, and the last two columns are $X_{\rm in}^\diamond$ obtained by Alg1 and Alg2, respectively}\label{visual-image1}
\end{figure}

\begin{figure}
	\centering
	\includegraphics[width=0.90\textwidth]{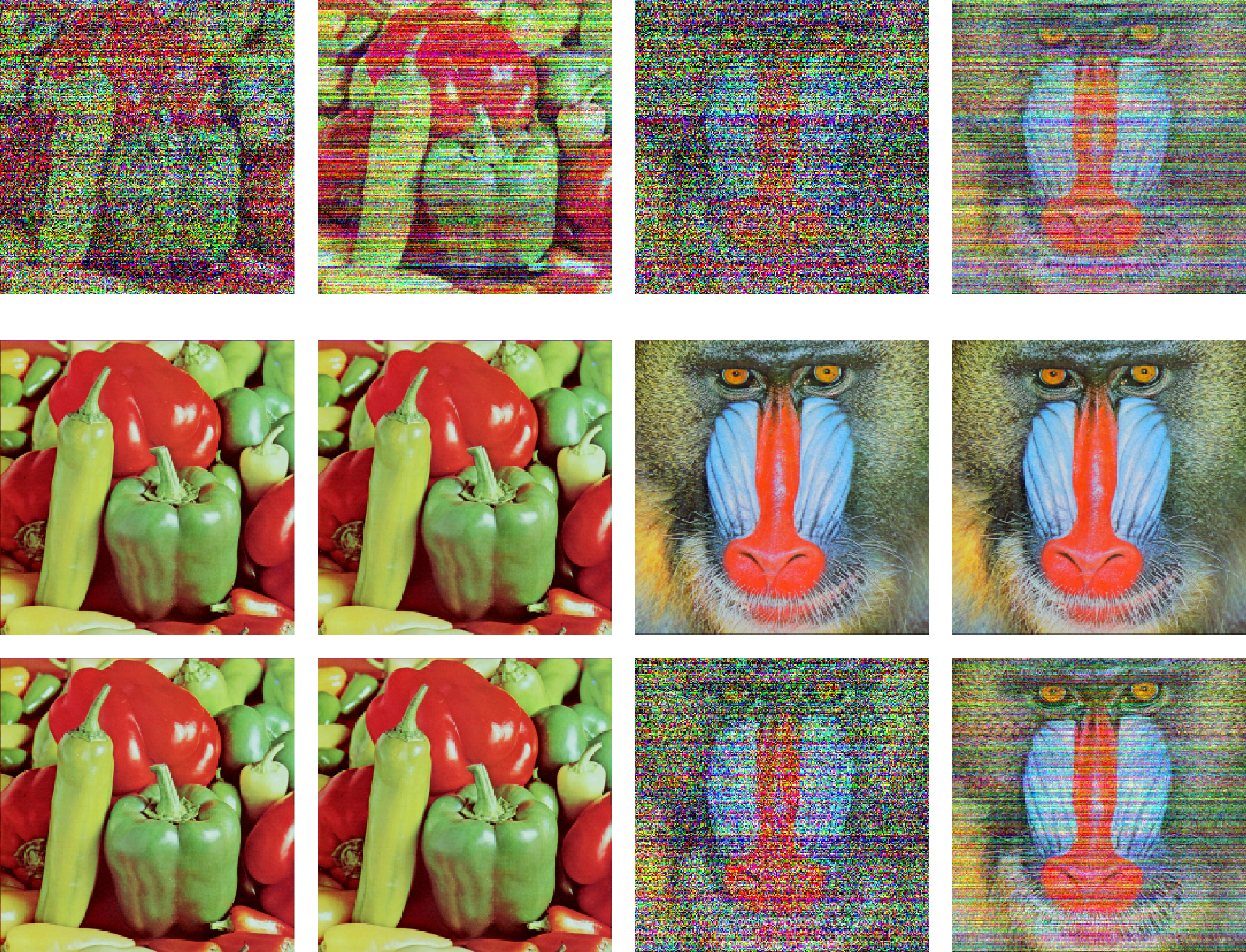}
		\caption{Visual results of $X_{\rm st}$ and $X_{\rm in}$ in ``peppers + baboon''. \\ First row: case (i) with $m=180$; Second row: case (ii) with $m=300$; Third row: case (ii) with $m=400$. The first two columns are $X_{\rm st}^\diamond$ obtained by Alg1 and Alg2, and the last two columns are $X_{\rm in}^\diamond$ obtained by Alg1 and Alg2, respectively}\label{visual-image2}
\end{figure}

\begin{figure}
	\centering
	\includegraphics[width=0.30\textwidth]{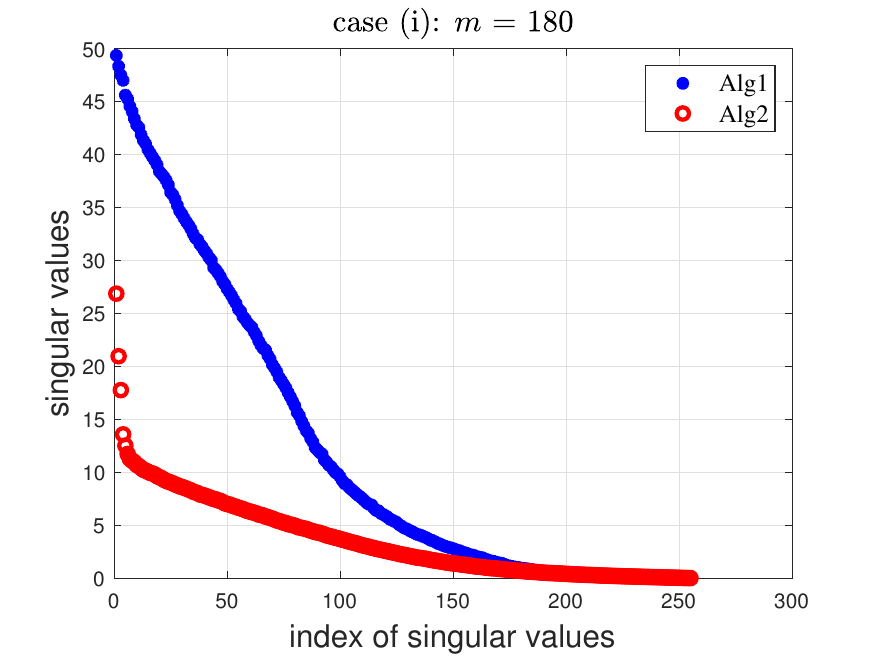}
	\includegraphics[width=0.30\textwidth]{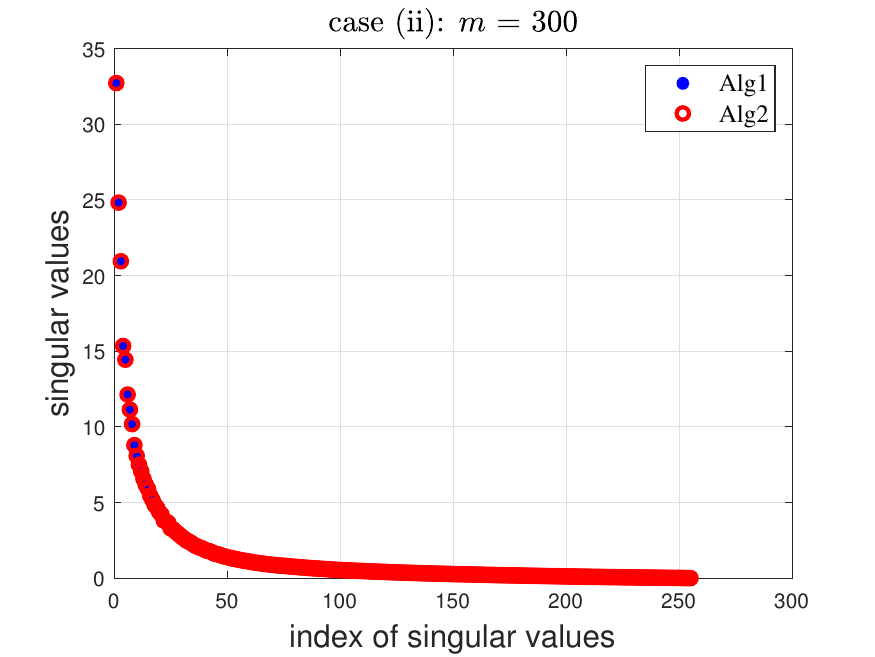}
	\includegraphics[width=0.30\textwidth]{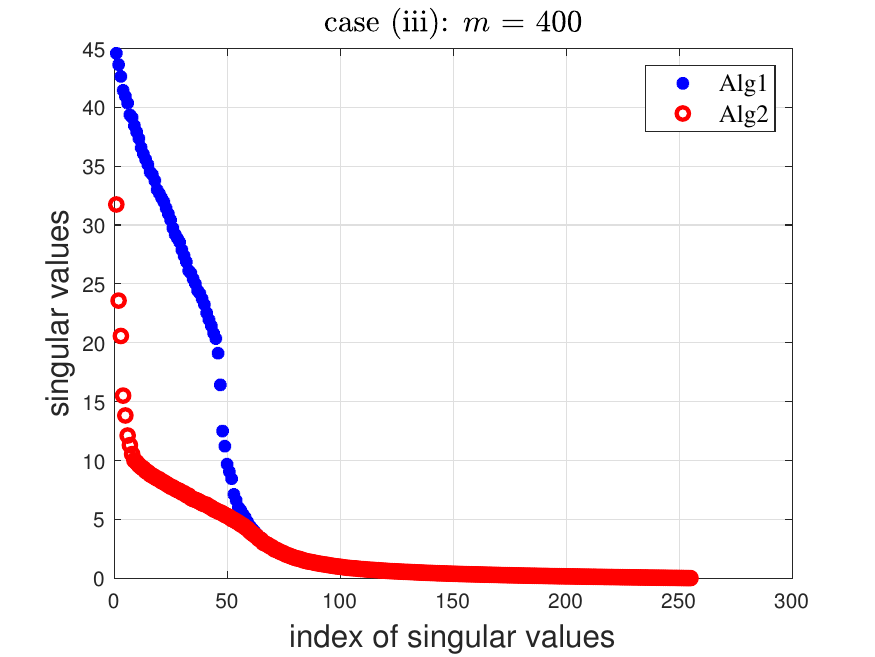}\\
	\includegraphics[width=0.30\textwidth]{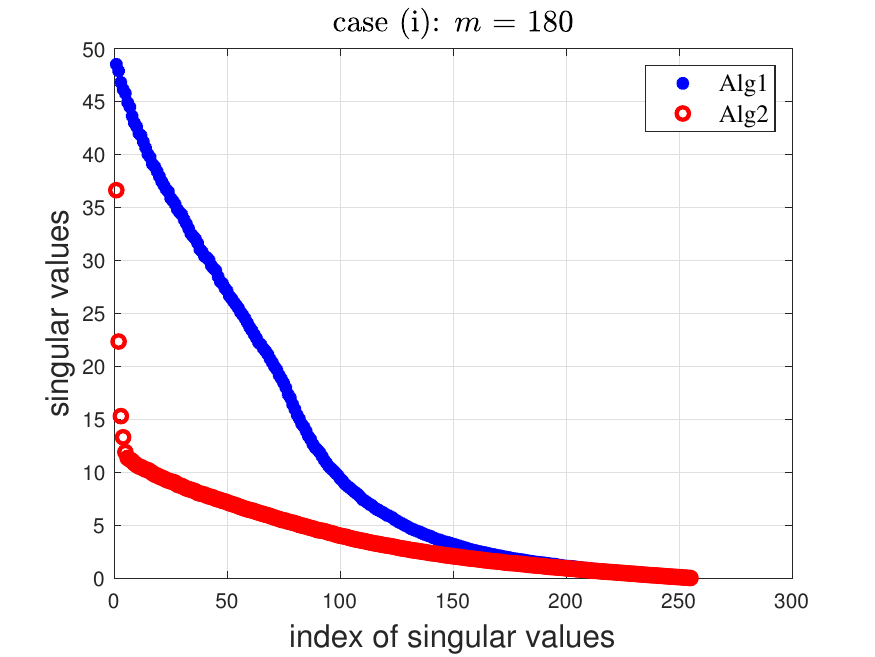}
	\includegraphics[width=0.30\textwidth]{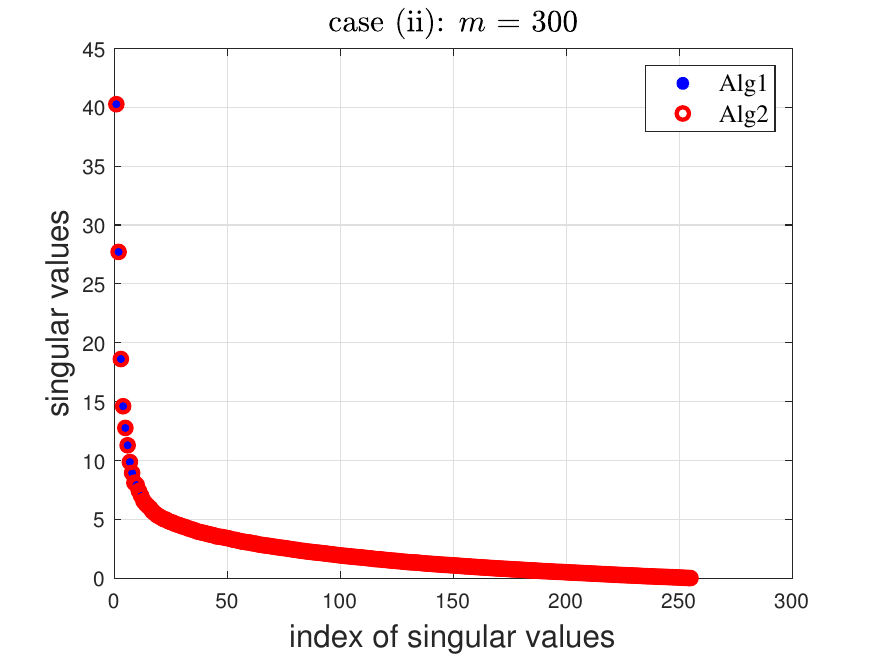}
	\includegraphics[width=0.30\textwidth]{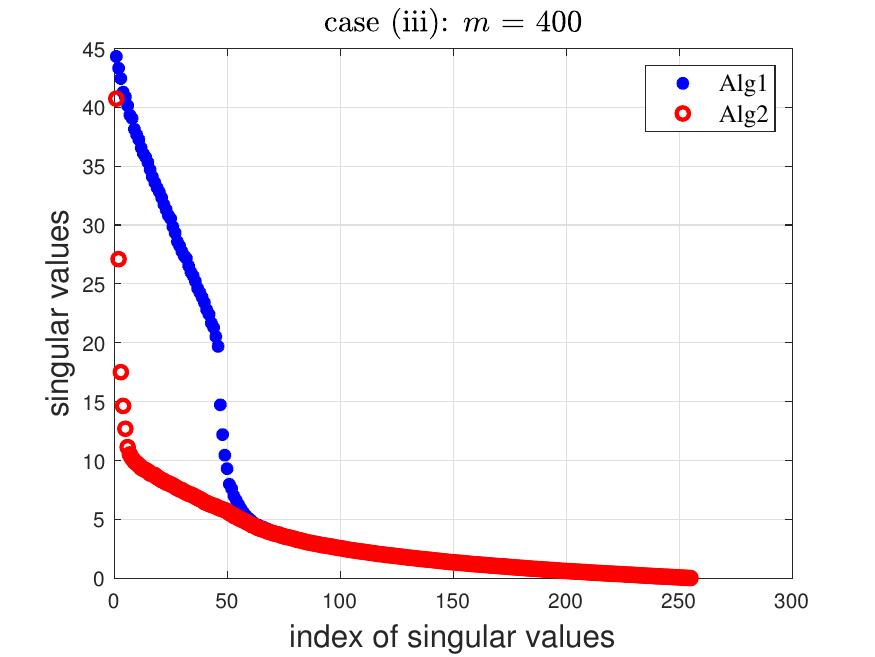}
	\caption{Results of singular values obtained by Alg1 and Alg2 in the images scenario}
\label{images-test-SVD}
\end{figure}

\section{Final remarks}\label{FinalRe}

In this paper, we proposed two approaches for solving least squares problem of dual quaternion equations, with the help of a newly introduced metric function. Firstly, the considered problems are converted into the corresponding bi-level optimization problems. Consequently, two proximal operator minimization algorithms are proposed. The first algorithm is used to solve the classical dual quaternion matrix least squares problem, while the second one is used to solve the least squares solution of the corresponding problem with regularization requirements (such as low-rankness and sparsity). Theoretically, we also showed the convergence properties of the presented algorithms. A series of numerical experiments on synthetic and color image datasets have demonstrated the effectiveness of the proposed approaches.

\bigskip
\noindent {\bf Acknowledgement}:
The authors would like to thank Dr. Hongjin He from Ningbo University for his valuable suggestions and discussions on the numerical algorithms and experiments.
The first author was supported in part by National Natural Science Foundation of China (No. 11971138).

\end{document}